\documentclass[a4paper,11pt]{amsart}
\pdfoutput=1
\usepackage[utf8]{inputenc}
\usepackage{amsmath, amssymb, amsthm}
\usepackage{mathrsfs}
\usepackage[all]{xy}
\usepackage[left=2.5cm,right=2.5cm,top=3cm,bottom=3cm,a4paper]{geometry}

\usepackage{tikz-cd}
\usepackage{tikz}
\usepackage{graphicx}
\usepackage{wrapfig}
\usepackage{subcaption}

\usepackage[english]{babel}
\usepackage[hyphens]{url}
\usepackage[bookmarks]{hyperref}
\hypersetup{
	linktocpage=true,
	pdfborderstyle={/S/S/W 1},
	bookmarksopen=true,
	bookmarksnumbered=true,
	breaklinks=true
}

\makeatletter
\newcommand{\addresseshere}{%
  \enddoc@text\let\enddoc@text\relax
}
\makeatother

\title[Lower semi-continuity of the Waldschmidt constants]{Lower semi-continuity of the Waldschmidt constants}

\author{Daseul Bae}
\address{Department of Mathematical Sciences, KAIST, 291 Daehak-ro, Yuseong-gu, Daejeon 305-701, Korea}
\email{bds0822@kaist.ac.kr}

\date{}

\begin{document}

\theoremstyle{plain}
\newtheorem{Thm}{Theorem}[section]
\newtheorem{Prop}[Thm]{Proposition}
\newtheorem{Cor}[Thm]{Corollary}
\newtheorem{Lem}[Thm]{Lemma}
\newtheorem*{Thm*}{Theorem}
\newtheorem*{Cor*}{Corollary}

\theoremstyle{definition}
\newtheorem{Def}[Thm]{Definition}
\newtheorem{Note}[Thm]{Notation}
\newtheorem{Exam}[Thm]{Example}
\newtheorem{Conj}[Thm]{Conjecture}

\newtheorem{Rmk}[Thm]{Remark}

\def\P{\mathbb{P}}
\def\Pic{\mathrm{Pic}}
\def\Cl{\mathrm{Cl}}
\def\Spec{\mathrm{Spec}\,}
\def\Proj{\mathrm{Proj}}
\def\Bl{\mathrm{Bl}}
\def\Sym{\mathrm{Sym}}
\def\Gr{\mathrm{Gr}}
\def\PProj{\mathbf{Proj}}
\def\SSym{\mathbf{Sym}}
\def\GGr{\mathbf{Gr}}
\def\O{\mathcal{O}}
\def\cL{\mathcal{L}}
\def\cE{\mathcal{E}}
\def\fW{\mathfrak{W}}
\def\fX{\mathfrak{X}}
\def\fY{\mathfrak{Y}}
\newcommand{\hooklongrightarrow}{\lhook\joinrel\longrightarrow}

\newcommand{\tens}[1]{%
  \mathbin{\mathop{\otimes}\displaylimits_{#1}}%
}



\begin{abstract}
In this paper, we study the Waldschmidt constant of a generalized fat point subscheme $Z=m_1p_1+\cdots+m_rp_r$ of $\P^2$, where $p_1,\cdots,p_r$ are essentially distinct points on $\P^2$, satisfying the proximity inequalities.
Furthermore, we prove its lower semi-continuity for $r\le 8$.
Using this property, we also calculate the Waldschmidt constants of the fat point subschemes $Z=p_1+\cdots+p_5$ giving weak del Pezzo surfaces of degree 4.
%
\end{abstract}


\maketitle

\section{Introduction.}\label{S1}

Let $k$ be an algebraically closed field of characteristic 0 and $\P^N$ be the projective space over $k$.
Let $R$ be the homogeneous coordinate ring of $\P^N$.
The \emph{Waldschmidt constant} $\widehat{\alpha}(I)$ of a homogeneous ideal $I\subset R$ asymptotically measures the degree of a hypersurface passing through the closed subscheme defined by $I$ in $\P^N$.
After Nagata's work for the 14th Hilbert problem, these constants received great attention.
Recently, Waldschmidt constants have been rediscovered by \cite{BH} in the containment relation between symbolic and ordinary powers of homogeneous ideals.
This result has renewed interest in computing $\widehat{\alpha}(I)$.
There are known results for some specific ideals:
Waldschmidt constants for planar point configurations are computed in \cite{FGH,BH6}.
Waldschmidt constants for a Stanley-Reisner ideal of certain classes of simplicial complexes are computed in \cite{BF,SS}.
Also the authors of \cite{BCG} compute Waldschmidt constants $\widehat{\alpha}(I)$ for squarefree monomial ideals $I$.
As a consequence, they show that $\widehat{\alpha}(I)$ can be expressed in terms of the fractional chromatic number of the hypergraph constructed from the primary decomposition of $I$.
However, computing Waldschmidt constants remains a difficult problem in general.
Instead, there has been great interest in finding a lower bound of the Waldschmidt constants.
It is proved in \cite{HS, MW} that
\[
\widehat{\alpha}(I) \ge \frac{\alpha(I)}{N}
\]
where $\alpha(I)$ is the minimal degree of the polynomials in $I$.
This inequality has been improved by \cite{GC}.
The author conjectured the following inequality, and proved it for $N=2$:
\begin{Conj}[{Chudnovsky \cite{GC}}]\label{Chud}
\[
\widehat{\alpha}(I) \ge \frac{\alpha(I)+N-1}{N}.
\]
\end{Conj}
\noindent Chudnovsky's conjecture is known for several cases: 
The conjecture for any finite set of points in $\P^2$ is proved in \cite{GC, HH}.
For $N\ge3$, it is proved in \cite{MD} that the inequality holds for any general points in $\P^3_k$,
and for any set of at most $N+1$ points in general position in $\P^N_k$ where $k$ is a field of characteristic 0.
Also the authors of \cite{JD, BH} proved the conjecture for any set of a binomial number of points in $\P^N$ forming a star configuration, and the authors of \cite{DT} proved it for any set of at most $2^n$ points in very general position in $\P^N$.
Recently, the authors of \cite{FMX} achieve a remarkable result on Chudnovsky's conjecture.
They show that the conjecture holds for any finite set of very general points in $\P^N_k$, where $k$ is an algebraically closed field of characteristic 0, which improves the result of \cite{DT}.
They also show that the conjecture holds for any finite set of points in $\P^N_k$ lying on a quadric, without any assumptions on $k$.
These results have been considerably improved in \cite{MSS}.

In this paper, we find a lower bound of the Waldschmidt constants in a different manner. 
We consider all configurations of $r$-points in $\P^2$ allowing infinitely near points, and examine how the Waldschmidt constant varies when we move the points to special position.
It turns out that the Waldschmidt constant of the given $r\le8$-points get smaller when we move the points to special position.

Let $p_1,\cdots,p_r\in\P^2$ be distinct points and let $m_1,\cdots,m_r$ be non-negative integers.
Consider a homogeneous ideal $I(Z)=\bigcap_{i=1}^r I(p_i)^{m_i}$, where each $I(p_i)$ denotes the homogeneous ideal defining the closed point $p_i$.
It defines a 0-dimensional subscheme of $\P^2$, called \emph{fat point subscheme} and denoted by $Z=m_1 p_1+\cdots+m_r p_r$.
As a graded $k$-algebra, $R=\bigoplus_{d\ge0} R_d$ where $R_d$ is the $k$-vector space spanned by the homogeneous polynomials of degree $d$.
The ideal $I(Z)=\bigoplus_{d\ge0}[I(Z)]_d$ is also a graded $k$-algebra with grading $[I(Z)]_d=I(Z)\cap R_d$.
\begin{Def}[\cite{MW}, cf. {\cite[Definition 1.3.2]{BH1}}]
\textit{The Waldschmidt constant} of a fat point sub-scheme $Z$ is the real number
\[
\widehat{\alpha}(Z) = \widehat{\alpha}(I(Z)) = \inf \left\{ \frac{d}{m} : \left[ I(Z)^{(m)} \right]_d\ne0, m>0 \right\},
\]
where $I(Z)^{(m)}$ denotes the $m$-th symbolic power of $I(Z)$.
\end{Def}
Note that $I(Z)^{(m)}=\bigcap_{i=1}^r I(p_i)^{mm_i} = I(mZ)$ and hence 
\[
\widehat{\alpha}(Z) = \inf_{m>0} \frac{\alpha\left( I(Z)^{(m)} \right)}{m} = \inf_{m>0} \frac{\alpha(I(mZ))}{m}.
\]
Let $b:X_{r}\rightarrow\P^2$ be the blowing up of $\P^2$ at the points $p_1,\cdots,p_r$ and let $X=X_r$.
Denote by $L$ the pull-back of the general line on $\P^2$ to $X$ and $E_i$ the pull-back of each exceptional divisor.
Let $E_Z = m_1E_1+\cdots+m_rE_r$ and $\mathscr{I}=b_*\mathcal{O}_X(-E_Z)$.
Note that $\mathscr{I}$ is a coherent sheaf of ideals on $\P^2$ defining $Z$,
and hence $I(Z)=\Gamma_*(\mathscr{I})$ where
\[
\Gamma_*(\mathscr{I}) := \bigoplus_{d\in\mathbb{Z}} \Gamma(\P^2,\mathscr{I}(d))=\bigoplus_{d\ge0} H^0(dL-E_Z)
\]
is the graded $R$-module associated to $\mathscr{I}$ (cf. \cite[Proposition IV.1.1]{BH3}).
The last equality is deduced from the projection formula.
More precisely, $I(Z)$ is the image of $\Gamma_*(\mathscr{I})$ under the natural inclusion $\Gamma_*(\mathscr{I}) \hooklongrightarrow \Gamma_*(\O_{\P^2})\cong k[\P^2]$
where the last isomorphism follows by \cite[Proposition II.5.13]{RH}.
Therefore, $[I(mZ)]_d \cong H^0(dL- mE_Z)$ as $k$-vector spaces for any integer $m>0$, and thus
\begin{align}\label{WaldConst1}
\widehat{\alpha}(Z) = \inf \left\{ \frac{d}{m} : H^0(dL- mE_Z)\ne0, m>0 \right\}.
\end{align}

Let $Z=m_1p_1+\cdots+m_rp_r$ and $Z'=m_1p_1'+\cdots m_rp_r'$ be fat point subschemes of $\P^2$ with generic $p_1,\cdots,p_r$.
Denote by $E_Z=m_1E_1+\cdots+m_rE_r$ and $E_{Z'}=m_1E_1'+\cdots+m_rE_r'$ the corresponding divisors.
By the upper semi-continuity of the cohomology groups, if $dL-mE_Z$ is effective, then so is $dL'-mE_{Z'}$, and hence $\widehat{\alpha}(Z)\ge \widehat{\alpha}(Z')$ (cf. \cite[Theorem I.1.6]{BH3}).
Note that the effective cones of the blowing up surfaces $X$ of $\P^2$ at generic points $p_1,\cdots, p_r$ are isomorphic, and thus $\widehat{\alpha}(Z)$ is constant for generic $p_1,\cdots, p_r$.
So we can expect that the Waldschmidt constant $\widehat{\alpha}(Z)$ has the biggest value at generic points $p_1,\cdots,p_r$, and it becomes smaller if we move the points to special position.
Such property also can be observed by some known results for Waldschmidt constants of planar point configurations, see e.g. \cite{FGH,BH6}.
Also see \cite{FMX,DT} for related works.

Now consider a family $S$ of a given points $p_1,\cdots,p_r$ (not necessarily generic).
The question is that there is an open subset $U\subset S$ on which $\widehat{\alpha}(Z)$ is constant.
If such $U$ always exists, we deduce the local minimality of $\widehat{\alpha}(Z)$ (\ref{MainThm1}) by induction on $\dim(S)$, and hence deduce our main result (\ref{MainThm2}).
If $r\le 8$, the effective cone of the blowing up surface $X$ is finitely generated (\ref{FiniteGen}), and it guarantees the existence of such $U$ (\ref{MainCor1}).

\begin{Thm*}[\ref{MainThm2}]
Let $S$ be an algebraic variety over $k$ and $\fX_{r}^S\to S$ an $r$-edpf (\ref{edpf}) with $r\le 8$.
Let $\cE_Z^S=m_1\cE_1^S+\cdots+m_r\cE_r^S$ with $m_i\ge0$ and suppose $-\cE_Z^s$ satisfies the proximity inequalities for all $s\in S$.
The function $\widehat{\alpha}_{Z^S}: S \rightarrow \mathbb{R}$ (\ref{ClosedSubschemeZ}) defined by $s\mapsto \widehat{\alpha}(Z^s)$ is a lower semi-continuous function on $S$.
Furthermore, the image of $\widehat{\alpha}_{Z^S}$ is a finite set.
\end{Thm*}

To state and prove the main theorem rigorously, we need to choose an appropriate moduli space of distinct points.
In order to fix the multiplicities $m_i$ of $Z$, we have to consider the order of the points.
Following \cite{SK}, we construct the universal family $\fX_{r-1}$ of $r$-points considering their order (\S\ref{S32}).
The family $\fX_{r-1}$ in fact parametrizes $r$-points on $\P^2$ including infinitely near points.

What we discussed so far is the behavior of the Waldschmidt constants of planar $r\le8$-distinct points.
However, the assumption that the points are distinct is not necessary.
In fact, \cite{BH2} generalizes the notion of fat point subscheme by using complete ideals,
and the theorem nicely extends to generalized fat point subscheme.
We will discuss these notions in \S\ref{S2}.
\\

\noindent
\textbf{The organization of the paper.}
In Section \ref{S2}, we discuss generalized fat point subschemes and Waldschmidt constants of them.
In Section \ref{S3}, we first introduce preliminaries for constructing the universal family of essentially distinct points $\fX_{r-1}$.
Further, we discuss the notion of a family of $r$-essentially distinct points and examine their related properties.
In Section \ref{S4}, we prove our main results: Theorem \ref{MainThm1} and Theorem \ref{MainThm2}.
As an application of our main theorems, we calculate the Waldschmidt constants of generalized fat point subschemes $Z=p_1+\cdots+p_5$, which give rise to weak del Pezzo surfaces of degree 4 in the last section.
\\

\noindent
\textbf{Acknowledgements.} I would like to thank my thesis advisor, Yongnam Lee, for his encouragement and great support.
Also I thank Brian Harbourne for his incentive lectures at KAIST while this work was carried out.
The author was supported by Basic Science Research Program through the NRF of Korea (2016930170).


\vspace{3em}
\section{Generalized fat point subscheme}\label{S2}

Let $X_0=\P^2$ and let $b_{i}:X_{i}\rightarrow X_{i-1}$ be the blowing up of $X_{i-1}$ at a point $p_i\in X_{i-1}$ for $i=1,\cdots,r$.
The points $p_1,\cdots, p_r$  are called \emph{$r$-essentially distinct points} of $X_0$.
Denote by $X=X_r$, $b_{r,i-1}=b_{r}\circ\cdots\circ b_{i}:X\rightarrow X_{i-1}$ the composition, and $b= b_{r,0}:X\rightarrow X_{0}$.
Denote also by $L$ the pull-back divisor of a general line on $\P^2$ to $X$,
and $E_i$ the pull-back divisor of the exceptional divisor of $b_{i}:X_{i}\rightarrow X_{i-1}$ to $X$.
The classes of $L,E_1,\cdots,E_r$ in $\Cl(X)$ are called the \textit{exceptional configuration} corresponding to $p_1\cdots,p_r$ (cf. \cite{BH2}).
Let $E = m_1E_1+\cdots+m_rE_r$, $m_i\ge0$, and let $\mathscr{I}=b_*\mathcal{O}_X(-E)$.
Then $\mathscr{I}$ is a \emph{complete} coherent sheaf of ideals on $\P^2$, that is,
its stalks $\mathscr{I}_x\subset \mathcal{O}_{\P^2,x}$ are complete ideals in $\mathcal{O}_{\P^2,x}$, either primary for $\mathfrak{m}_x$ (the unique maximal ideal of $\mathcal{O}_{\P^2,x}$) or $\mathscr{I}_x=\mathcal{O}_{\P^2,x}$, for every $x\in\P^2$ (\cite[Theorem 15]{BH4}).
It defines a $0$-dimensional subscheme $Z$ of $\P^2$, denoted by $Z=m_1p_1+\cdots+ m_rp_r$.
Conversely, if $\mathscr{I}$ is a complete coherent sheaf of ideals on $\P^2$, it defines a $0$-dimensional subscheme $Z$ of $\P^2$,
and there are essentially distinct points $p_1,\cdots,p_r$ of $\P^2$ such that $\mathscr{I}=b_*\mathcal{O}_X(-E_Z)$ for some $E_Z = m_1E_1+\cdots+m_rE_r$, $m_i\ge0$ (\cite[Corollary 17]{BH4}).
Furthermore, we can choose $-E_Z$ satisfying \textit{the proximity inequalities}: $-E_Z\cdot C\ge0$ for every irreducible component $C$ of $E_i$, $i=1,\cdots,r$.
Hence, there is a bijection between complete coherent sheaves of ideals $\mathscr{I}=b_*\mathcal{O}_X(-E_Z)$ and $-E_Z$ satisfying the proximity inequalities.
In this way, we generalize the usual notion of fat point subschemes by complete coherent sheaves of ideals on $\P^2$.
We also abbreviate saying that $-E_Z$ satisfies the proximity inequalities by simply saying that $Z$ satisfies the proximity inequalities.

\vspace{1em}

Now let $Z=m_1p_1+\cdots+m_rp_r$ be a fat point subscheme of $\P^2$ defined by a complete coherent sheaf of ideals $\mathscr{I}_Z=b_*\mathcal{O}_X(-E_Z)$, where $E_Z = m_1E_1+\cdots+m_rE_r$ with $m_i\ge0$.
Let $I=\Gamma_*(\mathscr{I}_Z)$. We define
\[
\widehat{\alpha}(Z) = \widehat{\alpha}(E_Z) = \inf \left\{ \frac{d}{m} : \left[ I^{(m)} \right]_d\ne0, m>0 \right\}.
\]
For given $m>0$, the least degree $\alpha(I^{(m)})$ is hard to predict.
However, if $[I^{(m)}]_d = H^0(dL-mE_Z)$, we have the same formula as \ref{WaldConst1}, and there are more advantages for computing $\widehat{\alpha}(Z)$, for example, Riemann-Roch theorem and semi-continuity of cohomology groups.
In fact, we have the equality if $Z$ satisfies the proximity inequalities.
First, we prove the following lemma.

\begin{Lem}\label{SymSat}
Let $I$ be as above. 
We have $(I^m)^{\mathrm{sat}} = I^{(m)}$, where $(\cdot)^\mathrm{sat}$ denotes the saturation of an ideal.
\end{Lem}

\begin{proof}
First, $(I^m)^{\mathrm{sat}} \subset I^{(m)}$ since $I^{(m)}$ is saturated.
Now let $M=I^{(m)}/(I^m)^{\mathrm{sat}}$ and suppose $M\ne0$.
Then there is a minimal prime $\mathfrak{p}\in\mathrm{Supp}(M)$.
Since $M\subset R/(I^m)^{\mathrm{sat}}$, $\mathfrak{p}\in \mathrm{Ass}(R/(I^m)^{\mathrm{sat}})$ the set of associated primes of $R/(I^m)^{\mathrm{sat}}$.
In particular, $\mathfrak{p}$ is a homogeneous ideal which is not the maximal homogeneous ideal $\mathfrak{m}$.
Note that $I\subset \mathfrak{p}\subsetneq \mathfrak{m}$, and therefore $\mathfrak{p}\in\mathrm{Ass}(R/I)$ since $I$ defines a $0$-dimensional subscheme of $\P^2$.
Then it follows that $I^{(m)}\subset I^mR_\mathfrak{p}$ in the field of fractions of $R$, and hence
\[
\left( (I^m)^{\mathrm{sat}} \right)_{\mathfrak{p}} = \left( I^m \right)_{\mathfrak{p}} \subset \left( I^{(m)} \right)_{\mathfrak{p}} = I^{(m)} R_{\mathfrak{p}} \subset I^mR_\mathfrak{p} = \left( I^m \right)_{\mathfrak{p}}.
\]
Therefore $\left( (I^m)^{\mathrm{sat}} \right)_{\mathfrak{p}}=\left( I^{(m)} \right)_{\mathfrak{p}}$, i.e., $M_{\mathfrak{p}}=0$ which is a contradiction.
\end{proof}

Denote by $\mathscr{I}_Z^{[m]} = b_*\mathcal{O}_X(-mE_Z)$.
Note that it depends on the choice of $E_Z$.
Denote also by $I^{[m]}=\Gamma_*(\mathscr{I}_Z^{[m]})$ the graded $R$-module associated to $\mathscr{I}_Z^{[m]}$.
Then $I^{[m]} \cong \bigoplus_{d\ge0} H^0(dL-mE_Z)$ by the projection formula.

\begin{Prop}\label{WaldConst}
Let $Z$ and $\mathscr{I}_Z$ be as before.
If $Z$ satisfies the proximity inequalities, the canonical homomorphism ${\mathscr{I}_Z}^m \to \mathscr{I}_Z^{[m]}$ is an isomorphism for all $m>0$.
Furthermore, in this case, $I^{(m)}= I^{[m]}$ for all $m>0$ and therefore, $\widehat{\alpha}(Z) = \inf \left\{ \frac{d}{m} : H^0(dL- mE_Z)\ne0, m>0 \right\}$.
\end{Prop}

\begin{proof}
Let $f_1,\cdots,f_r$ be the dual configuration of $E_1,\cdots,E_r$.
They form a basis of $\Cl(X)$ and generate all the classes of $-m_1E_1-\cdots-m_rE_r$ satisfying proximity inequalities (cf. \cite[Proposition 2]{BH4}),
and thus we can uniquely write $\mathscr{I}_Z=b_*\mathcal{O}_X(F_Z)$, where $F_Z=n_1f_1+\cdots+n_rf_r$ with $n_i\ge0$.
By \cite[Theorem 15]{BH4},
\begin{align*}
{\mathscr{I}_Z}^m \cong \left( b_*\mathcal{O}_X(F_Z) \right)^m 
&\cong \left( b_*\mathcal{O}_X(f_1)^{n_1}\cdots b_*\mathcal{O}_X(f_r)^{n_r} \right)^m \\
&\cong b_*\mathcal{O}_X(f_1)^{mn_1}\cdots b_*\mathcal{O}_X(f_r)^{mn_r} \cong b_*\mathcal{O}_X(mF_Z) = \mathscr{I}_Z^{[m]}
\end{align*}
for all $m\ge1$.
Therefore, for any $m\ge1$,
\[
I^{(m)}= \left( I^m \right)^{\mathrm{sat}} = \Gamma_*({\mathscr{I}_Z}^m) = I^{[m]}
\] 
as ideals in $R$ by \ref{SymSat}.
It follows that $\widehat{\alpha}(Z) = \inf \left\{ \frac{d}{m} : H^0(dL- mE_Z)\ne0, m>0 \right\}$.
\end{proof}

\begin{Exam}
Let $p_1\in\P^2$ and $p_2$ an infinitely near point of $p_1$.
Let $\pi:X=X_2\to X_1\to \P^2$ be the corresponding sequence of blowing ups.
Changing the coordinates if necessary, we may assume that $p_1=[0,0,1]$ and that $p_2$ is the intersection of the proper transformation of the line $(x=0)$ and the exceptional divisor $F_1$ of $X_1\to\P^2$. 
Then 
\[
X_1\cong\Proj(k[x,y,s,t]/(xt-ys))
\]
locally over $\mathbb{A}^2=\Spec k[x,y]$ and $p_2$ is the point defined by the ideal $(x,y,s)$.
Let $E:=E_1+2E_2$ where $E_1$ is the total transformation of $F_1$ on $X$, and $E_2$ is the exceptional divisor of the second blowing up. 
Note that $-E$ does not satisfies the proximity inequalities.
Let 
\[
\mathscr{I}:=\pi_*\O_{X}(-E) = \pi_*\O_X(-E_1-2E_2).
\]
Let $I=\Gamma_*(\mathscr{I})$ considered as an ideal in $R=k[\P^2]$.
From the direct calculation, $I=(x^2,xy,y^3)$, thus
\[
I^{(2)}= \left(I^2 \right)^{\mathrm{sat}}= (x^4,x^3y,x^2y^2,xy^4,y^6).
\]
On the other hand, 
\[
\mathscr{I}^{[2]}=\pi_*\O_X(-2E) = \pi_*\O_X(-2E_1-4E_2).
\]
We now compute $I^{[2]}$.
Let $\widehat{F}_1$ be the strict transformation of $F_1$ on $X$.
As a divisor, $\widehat{F}_1 = E_1 - E_2$.
Consider the following short exact sequence of sheaves:
\[
0 \longrightarrow \O_X(-3E_1-3E_2) \stackrel{\varphi}{\longrightarrow} \O_X(-2E_1-4E_2) \longrightarrow \O_{\widehat{F}_1}(-2) \longrightarrow 0
\]
where $\varphi$ is given by tensoring $\O_X(\widehat{F}_1)$.
Let $L\subset X$ be the total transformation of a general line on $\P^2$ so that $\O_X(dL)=\pi^*\O_{\P^2}(d)$ for all $d\ge0$.
Then for any $d\ge0$, tensoring $\pi^*\O_{\P^2}(d)$ on the above short exact sequence of sheaves yields the following short exact sequence of sheaves:
\[
0 \longrightarrow \O_X(dL-3E_1-3E_2) \longrightarrow \O_X(dL-2E_1-4E_2) \longrightarrow \O_{\widehat{F}_1}(-2) \longrightarrow 0.
\]
Taking the $\Gamma(X,-)$-functor, we have
\[
0 \longrightarrow \Gamma(X,\O_X(dL-3E_1-3E_2)) \stackrel{\simeq}{\longrightarrow} \Gamma(X,\O_X(dL-2E_1-4E_2)) \longrightarrow \Gamma(X,\O_{\widehat{F}_1}(-2)) = 0.
\]
By the projection formula,
\[
\Gamma(X,\O_X(dL-3E_1-3E_2)) \cong \Gamma(\P^2,\mathscr{J}^{[3]}(d))
\]
where $\mathscr{J} = \pi_*\O_X(-E_1-E_2)$, and
\[
\Gamma(X,\O_X(dL-2E_1-4E_2)) \cong \Gamma(\P^2,\mathscr{I}^{[2]}(d)).
\]
Hence there is a natural isomorphism $\Gamma(\P^2,\mathscr{J}^{[3]}(d)) \stackrel{\simeq}{\longrightarrow} \Gamma(\P^2,\mathscr{I}^{[2]}(d))$ for all $d\ge0$, induced by $\pi_*\varphi$.
It deduces the natural isomorphism
\[
\xymatrix{
\Gamma_*(\mathscr{J}^{[3]}) \ar[rr]^{\simeq}_{\Gamma_*(\pi_*\varphi)} &&
\Gamma_*(\mathscr{I}^{[2]}) \ar@{^(->}[r] & 
\Gamma_*(\O_{\P^2}) \cong R.
}
\]
Therefore $I^{[2]} = J^{[3]}$ where $J = \Gamma_*(\mathscr{J})$ in $R$.
Note that $J = (x, y^2)$ from direct calculation, and since $-E_1-E_2$ satisfies the proximity inequalities,
\[
J^{[3]} = J^{(3)} = (J^3)^{\mathrm{sat}} = (x^3,x^2y^2,xy^4,y^6)
\]
by \ref{WaldConst}.
Hence $I^{[2]} = (x^3,x^2y^2,xy^4,y^6)$ and therefore, $I^{(2)}\subsetneq I^{[2]}$ for $x^3\notin I^{(2)}$.
\end{Exam}


\vspace{3em}
\section{Universal family of essentially distinct points}\label{S3}


\subsection{Preliminaries}\label{S31}

Before introduce the universal family of essentially distinct points, we remark some general facts.
Let $S$ be a scheme over $k$ and $X$ an $S$-scheme.
Let $Y\subset X$ be a closed subscheme, and $\mathscr{I}$ the corresponding quasi-coherent sheaf of ideals on $X$.
Let $X'=\Bl(X,Y)$ be the blowing up of $X$ along $Y$, and $Y'=Y\times_X X'$ the total transformation of $Y$.
Finally, denote by $b:X'\to X$ the canonical morphism.

\begin{Prop}\label{Blowingup}
Assume the morphisms $X\to S$ and $Y\to S$ are smooth.
\begin{enumerate}
    \item The blowing up $b: X'\to X$ commutes with every base change $S_1\to S$.
    \item The morphism $b: X'\to X$ is projective.
    \item $X'$ and $Y'$ are smooth over $S$.
\end{enumerate}
\end{Prop}

\begin{proof}
By \cite[Th\'eor\`eme 17.12.1]{EGAIV4}, the closed immersion $i:Y\to X$ is regular, and hence we have the isomorphism
$\SSym_{\mathcal{O}_Y}^*(\mathcal{C}_{Y/X}) \cong \GGr_{\mathscr{I}}^*(\mathcal{O}_X)$ as graded $\mathcal{O}_Y$-algebras, where
\[
\GGr_{\mathscr{I}}^*(\mathcal{O}_X)=\bigoplus_{n\ge0} (\mathscr{I}^n/\mathscr{I}^{n+1}).
\]
In particular, $\GGr_{\mathscr{I}}^*(\mathcal{O}_X)$ is a locally free $\mathcal{O}_Y$-module.
Since $Y\to S$ is smooth, it is flat, and hence $\GGr_{\mathscr{I}}^*(\mathcal{O}_X)$ is flat over $S$ as an $\mathcal{O}_Y$-module.
Note that $\GGr_{\mathscr{I}}^*(\mathcal{O}_X)$ is supported on $Y$, 
so that it is flat over $S$ as an $\mathcal{O}_X$-algebra .
The first assertion follows from \cite[Proposition 19.4.8]{EGAIV4}.

Since both morphisms $X\to S$ and $Y\to S$ are locally of finite presentation, so is the closed immersion $i:Y\to X$ by \cite[Proposition 1.4.3(v)]{EGAIV1}.
In fact $i:Y\to X$ is of finite presentation since a closed immersion is quasi-compact and separated.
This implies that $\mathscr{I}$ is of finite type as an $\mathcal{O}_X$-module. 
Note that there is a canonical surjective morphism
\[
\SSym_{\mathcal{O}_X}^*(\mathscr{I}) \to \bigoplus_{n\ge0} \mathscr{I}^{n}.
\]
Therefore, $X'=\PProj(\bigoplus_{n\ge0} \mathscr{I}^{n})$ is a closed subscheme of the projective bundle 
$\P(\mathscr{I})=\PProj(\SSym_{\mathcal{O}_X}^*(\mathscr{I}))$ on $X$,
and therefore, the morphism $b:X'\to X$ is projective.

For the last part, note that, by \cite[Proposition 16.9.8]{EGAIV4}, the conormal sheaf $\mathcal{C}_{Y/X}=\mathscr{I}/\mathscr{I}^2$ is a locally free $\mathcal{O}_Y$-module of finite rank.
Therefore, 
\[
Y'= \PProj \left( \GGr_{\mathscr{I}}^*(\mathcal{O}_X) \right) 
\cong \PProj \left( \SSym_{\mathcal{O}_Y}^*(\mathcal{C}_{Y/X}) \right) = \P(\mathcal{C}_{Y/X})
\]
is smooth over $Y$ by \cite[Corollaire 17.3.9]{EGAIV4}.
Then it follows that $Y'$ is smooth over $S$.
Note that $b:X'\setminus Y' \to X\setminus Y$ is an isomorphism, and hence $X'$ is smooth over $S$ at every point $x'\in X'\setminus Y'$.
Finally, by \cite[Proposition 19.4.8]{EGAIV4}, $X'$ is smooth over $S$ at every point $x'\in Y'$.
\end{proof}

\begin{Lem}\label{RelEffDiv1}
Assume the hypothesis of \ref{Blowingup}.
Consider the inverse image ideal sheaf $\mathscr{I}^m \cdot\mathcal{O}_{X'}$ with $m\ge1$.
Then the quotient sheaf $\mathcal{O}_{X'}/\mathscr{I}^m \cdot\mathcal{O}_{X'}$ is flat over $S$, i.e., $mY'$ is a relative effective Cartier divisor on $X'/S$.
\end{Lem}

\begin{proof}
Let $x'\in X'$ be arbitrary, and let $x\in X$ and $s\in S$ be the images of $x'$.
Note that $\mathcal{O}_{X'}/\mathscr{I}^m \cdot\mathcal{O}_{X'}\cong b^*(\mathcal{O}_{X}/\mathscr{I}^m)$, thus
\[
\left( \mathcal{O}_{X'}/\mathscr{I}^m \cdot\mathcal{O}_{X'} \right)_{x'} \cong \mathcal{O}_{X',x'} \otimes_{\mathcal{O}_{X,x}} \mathcal{O}_{X,x}/\mathscr{I}^m_x
\cong \bigoplus_{n\ge0} \left( \mathscr{I}^n_x \otimes_{\mathcal{O}_{X,x}} \mathcal{O}_{X,x}/\mathscr{I}^m_x \right)
\cong \bigoplus_{n\ge0} \left( \mathscr{I}^n_x /\mathscr{I}^{n+m}_x \right)
\]
as $\mathcal{O}_{X,x}$-modules.
Therefore, $\left( \mathcal{O}_{X'}/\mathscr{I}^m \cdot\mathcal{O}_{X'} \right)_{x'}$ is flat over $\mathcal{O}_{S,s}$ if and only if
$\mathscr{I}^n_x /\mathscr{I}^{n+m}_x$ is flat over $\mathcal{O}_{S,s}$ for all $n\ge 0$.
It is obvious for $n=0$ by the hypothesis (see the proof of \ref{Blowingup}), and for $n\ge1$ the result follows from the exact sequence
\[
0 \longrightarrow \mathscr{I}^n_x /\mathscr{I}^{n+m}_x \longrightarrow  \mathcal{O}_{X,x}/\mathscr{I}^{n+m}_x \longrightarrow \mathcal{O}_{X,x}/\mathscr{I}^{n}_x \longrightarrow 0
\]
of $\mathcal{O}_{X,x}$-modules.
\end{proof}

\begin{Lem}\label{RelEffDiv2}
Assume the hypothesis of \ref{Blowingup}. Let $E$ be a relative effective Cartier divisor on $X/S$, and
denote by $E'$ the pull-back divisor of $E$ by the canonical morphism $b: X'\to X$.
Then $E'$ is also a relative effective Cartier divisor on $X'/S$.
\end{Lem}

\begin{proof}
Let $x'\in X'$ be arbitrary, and let $x\in X$ and $s\in S$ be the images of $x'$.
We need to check that $\mathcal{O}_{E',x'}$ is flat over $\mathcal{O}_{S,s}$.
Since
\[
\mathcal{O}_{E',x'} \cong \mathcal{O}_{X',x'}\otimes_{\mathcal{O}_{X,x}} \mathcal{O}_{E,x} 
\cong \bigoplus_{n\ge0} \left( \mathscr{I}^n_x \otimes_{\mathcal{O}_{X,x}} \mathcal{O}_{E,x} \right),
\]
$\mathcal{O}_{E',x'}$ is flat over $\mathcal{O}_{S,s}$ if and only if $\mathscr{I}^n_x \otimes_{\mathcal{O}_{X,x}} \mathcal{O}_{E,x}$ is flat over $\mathcal{O}_{S,s}$ for all $n\ge 0$.
It is obvious for $n=0$ by the hypothesis.
So assume $n\ge1$.
Let $\mathscr{J}$ be the ideal sheaf of $E$ on $X$.
Note that
\[
\mathscr{I}^n_x \otimes_{\mathcal{O}_{X,x}} \mathcal{O}_{E,x} \cong \mathscr{I}^n_x/\mathscr{J}_x \mathscr{I}^n_x,
\]
and that there is an exact sequence of $\mathcal{O}_{X,x}$-modules
\[
0 \longrightarrow \mathscr{I}^n_x/\mathscr{J}_x \mathscr{I}^n_x \longrightarrow \mathcal{O}_{X,x}/\mathscr{J}_x \mathscr{I}^n_x
\longrightarrow \mathcal{O}_{X,x}/\mathscr{I}^n_x \to 0.
\]
Here $\mathcal{O}_{X,x}/\mathscr{I}^n_x$ is flat over $\mathcal{O}_{S,s}$ by the assumption.
Therefore, $\mathscr{I}^n_x/\mathscr{J}_x \mathscr{I}^n_x$ is flat over $\mathcal{O}_{S,s}$ if and only if $\mathcal{O}_{X,x}/\mathscr{J}_x \mathscr{I}^n_x$ is flat over $\mathcal{O}_{S,s}$.
The latter follows by the exact sequence of $\mathcal{O}_{X,x}$-modules
\[
0 \longrightarrow \mathscr{J}_x/\mathscr{J}_x \mathscr{I}^n_x \longrightarrow \mathcal{O}_{X,x}/\mathscr{J}_x \mathscr{I}^n_x
\longrightarrow \mathcal{O}_{X,x}/\mathscr{J}_x \to 0.
\]
In fact, $\mathcal{O}_{X,x}/\mathscr{J}_x$ is flat over $\mathcal{O}_{S,s}$ since $E$ is flat over $S$.
Also $\mathscr{J}_x/\mathscr{J}_x \mathscr{I}^n_x \cong \mathscr{J}_x\otimes_{\mathcal{O}_{X,x}} \mathcal{O}_{X,x}/\mathscr{I}^n_x$,
and therefore it is flat over $\mathcal{O}_{S,s}$ since $\mathscr{J}_x\cong \mathcal{O}_{X,x}$ as $\mathcal{O}_{X,x}$-modules.
Hence $\mathcal{O}_{X,x}/\mathscr{J}_x \mathscr{I}^n_x$ is flat over $\mathcal{O}_{S,s}$.
\end{proof}


\vspace{1em}
\subsection{Universal family of essentially distinct points}\label{S32}

We first introduce the parametrized space of essentially distinct points constructed by \cite{SK}, and its properties.
Let $X$ be a non-singular projective variety over $k$.
For $i\ge-1$, define a sequence of varieties $\fX_i$ and morphisms $\pi_{i+1},\, \xi_{i+1}:\fX_{i+1}\to \fX_{i}$ recursively as follows.
Let $\fX_{-1}=\Spec k$, $\fX_{0}=X$ and $\pi_{0}:\fX_{0}\rightarrow \fX_{-1}$ be the structure morphism.
Let $\xi_{0}=\pi_{0}$.
Fix $i\ge0$. Suppose we already constructed $\fX_{i-1},\, \fX_i$, and two morphisms $\pi_{i},\, \xi_{i}$.
Let $\fY_i$ be the fiber product as in the following diagram.
\begin{align*}
\xymatrix{
\fY_i=\fX_i\times_{\fX_{i-1}} \fX_i \ar[d]_{\mathrm{Pr}_{i,1}} \ar[rr]^(.65){\mathrm{Pr}_{i,2}} && \fX_i \ar[d]^{\pi_{i}} \\
\fX_i \ar[rr]_{\pi_{i}} && \fX_{i-1}
}
\end{align*}
Let $\Delta_i:\fX_i\to \fY_i$ be the diagonal morphism, and $\fX_{i+1}=\Bl(\fY_i,\Delta_i(\fX_i))$ the blowing up of $\fY_i$ along $\Delta_i(\fX_i)$.
Let $\pi_{i+1}:\fX_{i+1} \stackrel{b_{i+1}}{\longrightarrow} \fY_i \stackrel{\mathrm{Pr}_{i,1}}{\longrightarrow} \fX_i$ and $\xi_{i+1}:\fX_{i+1} \stackrel{b_{i+1}}{\longrightarrow} \fY_i \stackrel{\mathrm{Pr}_{i,2}}{\longrightarrow} \fX_i$ be the compositions.
Recursively we have the following commutative diagram.
\begin{align}\label{Di1}
\xymatrix{
\cdots\ar[r]^{\xi_{i+2}} & \fX_{i+1} \ar[d]_{\pi_{i+1}} \ar[r]^{\xi_{i+1}} & \fX_i \ar[d]^{\pi_{i}}\ar[r]^{\xi_{i}} & \cdots\ar[r]^{\xi_2} & \fX_1 \ar[r]^{\xi_1}\ar[d]^{\pi_1} & \fX_0\ar[d]^{\pi_{0}} \\
\cdots\ar[r]_{\pi_{i+1}} & \fX_i \ar[r]_{\pi_{i}} & \fX_{i-1} \ar[r]_{\pi_{i-1}} & \cdots\ar[r]_{\pi_{1}} & \fX_0 \ar[r]_{\pi_{0}} & \fX_{-1}
}
\end{align}

%

By \ref{Blowingup}, every morphism in the diagram \ref{Di1} is projective and smooth.
The crucial fact is that the blowing up $b_{i+1}: \fX_{i+1}\to \fY_i$ commutes with every base change $f:S\to \fX_{i}$ (\ref{Blowingup}), that is, if we have the following Cartesian diagram
\begin{align*}
\xymatrix{
S\ar[d]_{f^*\Delta_i}\ar[r] & \fX_i\ar[d]^{\Delta_i} \\
S\times_{\fX_i} \fY_i \ar[r]\ar[d]_{f^*\mathrm{Pr}_{i,1}} & \fY_i \ar[d]^{\mathrm{Pr}_{i,1}} \ar[r]^(.5){\mathrm{Pr}_{i,2}} & \fX_i \ar[d]^{\pi_{i}} \\
S\ar[r]_{f} & \fX_i \ar[r]_{\pi_{i}} & \fX_{i-1}
}
\end{align*}
the canonical morphism $\Bl(S\times_{\fX_i}\fY_i, f^*\Delta_i(S)) \longrightarrow S\times_{\fX_i} \fX_{i+1}$ is an isomorphism.

We now define the notion of a family of $r$-essentially distinct points of $X$ (cf. \cite[\S2]{PM}).
\begin{Def}\label{edpf}
Let $r\ge1$.
A \textit{family of $r$-essentially distinct points of $X$}
(or an \textit{$r$-edpf of $X$})
is a couple $(S,\sigma_\bullet)$ consisting of
a $k$-scheme $S$ and a finite sequence $\sigma_\bullet$ of morphisms $\sigma_1,\cdots,\sigma_{r}$ where
each $\sigma_i$ is a section of the morphism $\pi_{i-1}^S:\fX_{i-1}^S\to S$ defined recursively as follows.
Define $\fX_{0}^S=S\times_k X$ and $\pi_{0}^S:\fX_{0}^S\to S$ to be the projection $S\times_k X \to S$.
For $i\ge1$, let $\fX_{i}^S=\Bl(\fX_{i-1}^S,\sigma_{i}(S))$ be the blowing up and define
$\pi_{i}^S:\fX_{i}^S\to S$ to be the composition 
$\pi_{i-1}^S \circ b_{i}^S$, where $b_{i}^S:\fX_{i}^S \to \fX_{i-1}^S$ is the canonical morphism.
\end{Def}

\begin{Note}\label{edpfNote}
For convenience, let us also call $\pi_{r}^S: \fX_{r}^S\to S$ an $r$-edpf. 
So $\pi_{0}^S: S\times_k X\to S$ is the family of 0-essentially distinct points of $X$ over $S$.
\end{Note}


\begin{Rmk}
Again by \ref{Blowingup}, every $\pi_{i}^S:\fX_{i}^S\to S$ is projective and smooth.
Also, every $b_{i}^S : \fX_i^S\to \fX_{i-1}^S$ commutes with every base change $T\to S$.
\end{Rmk}

The notation $\fX$ in \ref{edpf} and \ref{edpfNote} is closely related to the varieties $\fX_{i}$ in the diagram \ref{Di1}.
In fact for each $r\ge0$, there is an $r$-edpf $(\fX_{r-1},\alpha_{r-1,\bullet})$ such that
the induced morphism $\pi_{r}^{\fX_{r-1}}:\fX_{r}^{\fX_{r-1}}\to \fX_{r-1}$ defined in \ref{edpf} is isomorphic to the morphism $\pi_r:\fX_r \to \fX_{r-1}$ in the diagram \ref{Di1}.
That is, there is an isomorphism $\fX_{r}^{\fX_{r-1}} \cong \fX_r$ and the two morphisms $\pi_{r}^{\fX_{r-1}},\, \pi_r$ commute with the isomorphism.
Moreover, for any $r$-edpf $(S,\sigma_\bullet)$, there exists a unique morphism $f:S\to \fX_{r-1}$ so that $\sigma_\bullet$ can be constructed from $f$ and $\alpha_{r-1,\bullet}$.
This is the main result of this section (\ref{UnivEdpf}), and this structural property plays a significant role in next two sections.
We first construct the $r$-edpf $(\fX_{r-1},\alpha_{r-1,\bullet})$, and define the notion of a morphism between two $r$-edpfs.

\begin{Exam}\label{UnivEdpfConst}
Denote by $\pi_{j,i}:\fX_{j}\to \fX_{i}$ the composition $\pi_{j}\circ\cdots\circ\pi_{i+1}$ for $i<j$, and let $\pi_{j,j}=\mathrm{Id}_{\fX_{j}}$.
For $1\le i < r$, we have the following Cartesian diagram
\begin{align*}
    \xymatrix{
    \fX_{r-1} \ar[d]_{{\alpha_{r-1,i} \,=\, \pi_{r-1,i-1}}^*\Delta_{i-1}} \ar[r]^{\pi_{r-1,i-1}} & \fX_{i-1} \ar[d]^{\Delta_{i-1}} \\
    \fX_{r-1} \times_{\fX_{i-1}} \fY_{i-1} \ar[r] \ar[d]_{{\pi_{r-1,i-1}}^*\mathrm{Pr}_{i-1,1}} & \fY_{i-1} \ar[d]^{\mathrm{Pr}_{i-1,1}} \ar[r]^(.5){\mathrm{Pr}_{i-1,2}} & \fX_{i-1} \ar[d]^{\pi_{i-1}} \\
    \fX_{r-1} \ar[r]_{\pi_{r-1,i-1}} & \fX_{i-1} \ar[r]_{\pi_{i-1}} & \fX_{i-2}
    }
\end{align*}
Let $\alpha_{r-1,i}={\pi_{r-1,i-1}}^*\Delta_{i-1}$.
Note that $\fX_{r-1} \times_{\fX_{i-1}} \fY_{i-1} \cong \fX_{r-1} \times_{\fX_{i-2}} \fX_{i-1}$.
Therefore, considering $\alpha_{r-1,i}$ as a section of the projection $\fX_{r-1} \times_{\fX_{i-2}} \fX_{i-1} \to \fX_{r-1}$, 
we have 
\[
\Bl(\fX_{r-1} \times_{\fX_{i-2}} \fX_{i-1}, \alpha_{r-1,i}(\fX_{r-1})) \cong \fX_{r-1} \times_{\fX_{i-1}} \fX_{i}
\]
by \ref{Blowingup}.
Let $\alpha_{r-1,r}:\fX_{r-1} \to \fX_{r-1} \times_{\fX_{r-2}} \fX_{r-1}$ be the diagonal morphism $\Delta_{r-1}$.
Then $(\fX_{r-1},\alpha_{r-1,\bullet})$ is an $r$-edpf of $X$.
\end{Exam}

\begin{Def}
Let  $(S,\sigma_\bullet)$ and $(T, \tau_\bullet)$ be $r$-edpfs. 
An \textit{$r$-edpf-morphism} $f:(T,\tau_\bullet) \to (S, \sigma_\bullet)$ between two $r$-edpfs is a morphism of $k$-schemes $f:T\to S$ together with the morphisms $f_i:\fX_{i}^T \to \fX_{i}^S$, $i=0,\cdots, r$, defined recursively in such a way that we have the following commutative diagrams:
\begin{align*}
    \xymatrix{
    T \ar[d]_{\tau_i}\ar[r] & S\ar[d]^{\sigma_i} \\
    \fX_i^T \ar[r]^{f_i}\ar[d]_{\pi_{i}^T} & \fX_i^S \ar[d]^{\pi_{i}^S} \\
    T\ar[r]_{f} & S
    }
\end{align*}
First define $f_0$ to be the morphism $f\times_k\mathrm{Id}_{X}:T\times_k X \to S\times_k X$.
Assume $f_i$ is already defined.
Under the hypothesis, the above diagram is in fact Cartesian (\cite[Lemma 1.3]{PM}), 
and hence we have the canonical morphism 
\[
\fX_{i+1}^T\cong T\times_S \fX_{i+1}^S \stackrel{\mathrm{Pr_2}}{\longrightarrow} \fX_{i+1}^S,
\]
which is defined to be $f_{i+1}$.
\end{Def}

\begin{Prop}[Universal $r$-edpf]\label{UnivEdpf}
The $r$-edpf $(\fX_{r-1},\alpha_{r-1,\bullet})$ defined in \ref{UnivEdpfConst} satisfies the following universal property: 
for any $r$-edpf $(S,\sigma_\bullet)$, there exists a unique $r$-edpf-morphism $f:(S,\sigma_\bullet) \to (\fX_{r-1},\alpha_{r-1,\bullet})$.
In this sense, we call $(\fX_{r-1},\alpha_{r-1,\bullet})$ the universal $r$-edpf.
\end{Prop}

\begin{proof}
Let $g_1:S\to \fX_0$ be the composition $S\stackrel{\sigma_1}{\longrightarrow} \fX_0^S=S\times_k X \stackrel{\mathrm{Pr_2}}{\longrightarrow} X$ and $g_{0}=\pi_0\circ g_1$.
Consider the following commutative diagram:
\begin{align*}
    \xymatrix{
    S \ar[d]_{\sigma_1}\ar[r]^{g_1} & \fX_0\ar[d]^{\Delta_0} \\
    \fX_0^S= S\times_{\fX_{-1}} \fX_0 \ar[r]^(.6){}\ar[d]_{\pi_0^S={g_1}^*\mathrm{Pr}_{0,1}} & \fY_0 \ar[d]^{\mathrm{Pr}_{0,1}} \ar[r]^(.5){\mathrm{Pr}_{0,2}} & \fX_0 \ar[d]^{\pi_{0}} \\
    S\ar[r]_{g_1} & \fX_0 \ar[r]_{\pi_{0}} & \fX_{-1}
    }
\end{align*}
Again by \cite[Lemma 1.3]{PM}, the above diagram is Cartesian, and therefore by \ref{Blowingup}, we have the Cartesian diagram:
\begin{align*}
    \xymatrix{
    \fX_1^S\cong S\times_{\fX_0} \fX_1 \ar[r]^(.6){}\ar[d]_{\pi_1^S} & \fX_1 \ar[d]^{\pi_1} \ar[r]^(.5){\xi_1} & \fX_0 \ar[d]^{\pi_{0}} \\
    S\ar[r]_{g_1} & \fX_0 \ar[r]_{\pi_{0}} & \fX_{-1}
    }
\end{align*}
Let $g_2:S\to \fX_1$ be the composition $S\stackrel{\sigma_2}{\longrightarrow} \fX_1^S \cong S\times_{\fX_0} \fX_1 \stackrel{\mathrm{Pr_2}}{\longrightarrow} \fX_1$.
Note that $g_1=\pi_1\circ g_2$.
Thus, we again have the following Cartesian diagram
\begin{align*}
    \xymatrix{
    S \ar[d]_{\sigma_2}\ar[r]^{g_2} & \fX_1 \ar[d]^{\Delta_1} \\
    \fX_1^S\cong S\times_{\fX_0} \fX_1 \ar[r] \ar[d]_{\pi_1^S} & \fY_1 \ar[d]^{\mathrm{Pr}_{1,1}} \ar[r]^(.5){\mathrm{Pr}_{1,2}} & \fX_1 \ar[d]^{\pi_{1}} \\
    S\ar[r]_{g_2} & \fX_1 \ar[r]_{\pi_{1}} & \fX_{0}
    }
\end{align*}
and thus $\fX_2^S\cong S\times_{\fX_1} \fX_2$.
Inductively, we have a morphism $g_i:S\to \fX_{i-1}$ such that $g_{i-1}=\pi_{i-1}\circ g_i$, 
and $\fX_{i}^S \cong S \times_{\fX_{i-1}} \fX_{i}$ for $i=1,\cdots, r$.
Let $f= g_r$. By the construction, for any $i=1,\cdots, r$,
\begin{align}\label{Di2}
    \xymatrix{
    S \ar[d]_{\sigma_i}\ar[rr]^{f} && \fX_{r-1} \ar[d]^{\alpha_{r-1,i}} \ar[r]^{\pi_{r-1,i-1}} & \fX_{i-1} \ar[d]^{\Delta_{i-1}} \\
    \fX_{i-1}^S\cong S\times_{\fX_{i-2}} \fX_{i-1} \ar[rr]^{f_{i-1}} \ar[d]_{\pi_{i-1}^S} && \fX_{i-1}^{\fX_{r-1}}\cong \fX_{r-1} \times_{\fX_{i-2}} \fX_{i-1} \ar[d]^{\pi_{i-1}^{\fX_r}} \ar[r] & \fY_{i-1} \ar[r]^{\mathrm{Pr}_{i-1,2}} \ar[d]^{\mathrm{Pr}_{i-1,1}} & \fX_{i-1} \ar[d]^{\pi_{i-1}} \\
    S\ar[rr]_{f} && \fX_{r-1} \ar[r]_{\pi_{r-1,i-1}} & \fX_{i-1} \ar[r]_{\pi_{i-1}} & \fX_{i-2}
    }
\end{align}
is a Cartesian diagram, and $f_i:\fX_{i}^S \to \fX_{i}^{\fX_{r-1}}$ is isomorphic to the  canonical morphism
\[
\fX_{i}^S\cong S\times_{\fX_{r-1}} \fX_{i}^{\fX_{r-1}} \stackrel{\mathrm{Pr_2}}{\longrightarrow} \fX_{i}^{\fX_{r-1}}.
\]
Therefore $f:(S,\sigma_\bullet) \to (\fX_{r-1},\alpha_{r-1,\bullet})$ is an $r$-edpf-morphism.

The uniqueness also follows from the diagram \ref{Di2}.
Suppose that such $f:S\to \fX_{r-1}$ exists and that the morphism $\pi_{r-1,i-2}\circ f:S\to \fX_{i-2}$ (the bottom row) is uniquely determined.
Chasing the diagram \ref{Di2}, we have 
\[
\pi_{r-1,i-1}\circ f = (\mathrm{Pr}_{i-1,2}\circ \Delta_{i-1}) \circ \pi_{r-1,i-1}\circ f = \mathrm{Pr}_{i-1,2}^S\circ\sigma_i,
\]
where $\mathrm{Pr}_{i-1,2}^S:\fX_{i-1}^S\to \fX_{i-1}$ is the composition of the morphism in the middle row of the diagram \ref{Di2}, 
i.e., the projection on the second factor.
Since $\pi_{r-1,i-2}\circ f$ is uniquely determined, 
so is $\mathrm{Pr}_{i-1,2}^S\circ\sigma_i$ because $\mathrm{Pr}_{i-1,2}^S$ is the pull-back of $\pi_{r-1,i-2}\circ f$ by $\pi_{i-1}$.
Thus $\pi_{r-1,i-1}\circ f$ is uniquely determined.
Note that $\pi_{r-1,-1}\circ f:S\to\Spec k$ is the structure morphism.
Obviously, it is uniquely determined.
Therefore, we conclude our claim by the induction on $i$.
\end{proof}

\begin{Rmk}
For $i=r+1$, we still have the following Cartesian diagram:
\begin{align*}
    \xymatrix{
    \fX_{r}^S \cong S\times_{\fX_{r-1}} \fX_{r} \ar[rr]^(.55){f_r} \ar[d]_{\pi_r^S} && \fX_{r}^{\fX_{r-1}} \cong \fX_{r} \ar[d]^{\pi_r^{\fX_{r-1}} = \pi_{r}} \\
    S\ar[rr]_{f} && \fX_{r-1}
    }
\end{align*}
Also we have the following commutative diagram:
\begin{align*}
\xymatrix{
\fX_{r}^S \ar[d]_{f_r} \ar[r]^{b_r^S} & \fX_{r-1}^S \ar[d]^{f_{r-1}} \ar[r]^{b_{r-1}^S} & \cdots\ar[r]^{b_1^S} & \fX_0^S \ar[r]^{\pi_0^S}\ar[d]^{f_{0}} & S \ar[d]^{} \\
\fX_{r} \ar[r]_{\xi_{r}} & \fX_{r-1} \ar[r]_{\xi_{r-1}} & \cdots\ar[r]_{\xi_{1}} & \fX_0 \ar[r]_{\xi_{0}} & \fX_{-1}
}
\end{align*}
where the columns are the projections on the second factors.
\end{Rmk}


Denote by $\xi_{j,i}:\fX_{j}\to \fX_{i}$ the composition $\xi_{j}\circ\cdots\circ \xi_{i+1}$ for $i<j$, and $\xi_{j,j}=\mathrm{Id}_{\fX_{j}}$.
Let $L$ be an effective Cartier divisor on $X$ and $\cL$ the pull-back divisor of $L$ to $\fX_{r}$ by $\xi_{r,0}$.
For $1\le i\le r$, denote by $E_i$ the exceptional divisor of the blowing up $b_i:\fX_{i}\to \fY_{i-1}$.
For $1\le i<r$, denote also by $\cE_i$ the pull-back divisor of $E_i$ to $\fX_{r}$ by $\xi_{r,i}$ and let $\cE_r=E_r$.
Note that $\cL$ and $\cE_i$ are relative effective Cartier divisors on $\fX_{r}/\fX_{r-1}$ (see \ref{RelEffDiv1} and \ref{RelEffDiv2}).

Similarly, denote by $b_{j,i}^S=b_{j}^S\circ\cdots\circ b_{i+1}^S$ for $i<j$ and $b_{j,j}^S=\mathrm{Id}_{\fX_j^S}$.
 Let $\cL^S$ be the pull-back divisor of $L$ to $\fX_{r}^S$ by the morphism $f_0\circ b_{r,0}^S:\fX_r^S\to X$.
Let $E_i^S$ be the exceptional divisor of the blowing up $b_i^S:\fX_{i}^S\to \fX_{i-1}^S$, and
$\cE_i^S$ the pull-back divisor of $E_i^S$ to $\fX_{r}^S$ by $b_{r,i}^S$.
Note that $\cL^S$ (resp., $\cE_i^S$) is in fact the pull-back divisor of $\cL$ (resp., $\cE_i$) by the base change $S\to \fX_r$.
In particular, $\cL^S$ and $\cE_i^S$ are relative effective Cartier divisors on $\fX_{r}^S$ over $S$ (\ref{RelEffDiv1}, \ref{RelEffDiv2}).
What we deduced is the following.

\begin{Prop}\label{RelEffDiv3}
Let $\fX_{r}^S\to S$ be an $r$-edpf, and let $\cE^S = m_1\cE_1^S+\cdots+m_r\cE_r^S$ for non-negative integers $m_i$.
Then $\cE^S$ is a relative effective Cartier divisor on $\fX_{r}^S/S$, 
and we have an exact sequence of flat $\mathcal{O}_S$-modules:
\[
0 \longrightarrow \mathcal{O}_{\fX_{r}^S}\left(-\cE^S\right) \longrightarrow \mathcal{O}_{\fX_{r}^S} \longrightarrow \mathcal{O}_{\cE^S} \longrightarrow 0.
\]
\end{Prop}


\vspace{3em}
\section{Lower semi-continuity theorems}\label{S4}

In this section, $X=\P^2$, $S$ is an algebraic variety over $k$, and we only consider those invertible sheaves on $\fX_{r}^S$ generated by $\cL^S$ and $\cE^S_i$.

\subsection{Effective cones of blowing up surfaces of $\P^2$}

\begin{Note}
Let $\fX_{r}^S\to S$ be an $r$-edpf.
Let $\mathscr{F}$ be an invertible sheaf on $\fX_{r}$.
Denote by $\mathscr{F}^S$ the pull-back sheaf of $\mathscr{F}$ to $\fX_{r}^S$ by the canonical morphism $\fX_{r}^S\to \fX_{r}$.
Note that any invertible sheaf on $\fX_{r}^S$ is of this form.
\end{Note}

\begin{Note}
For any (closed) point $s\in S$, we have a morphism of schemes $\Spec k(s)\to S$.
We denote by $\fX_{r}^s$ the induced $r$-edpf $\fX_r^{\Spec k(s)}\to \Spec k(s)$.
\end{Note}

\begin{Lem}\label{MainLem1}
Let $\fX_{r}^S\to S$ be an $r$-edpf and
$\mathscr{F}^S$ an invertible sheaf on $\fX_{r}^S$.
If $\mathscr{F}^s$ is effective for general $s\in S$, then $\mathscr{F}^s$ is effective for all $s\in S$.
\end{Lem}

\begin{proof}
Note that effectivity of $\mathscr{F}^s$ is equivalent to saying that $\dim_k H^0(\fX_{r}^s, \mathscr{F}^s)\ge 1$,
and this follows from the upper semi-continuity theorem for the cohomology groups.
\end{proof}

\begin{Rmk}\label{Marking}
For any closed point $s\in S$, $\cL^s$ and $\cE_i^s$ form a basis of the divisor class group $\Cl(\fX_{r}^s)$ of $\fX_{r}^s$.
In this way, we have a canonical isomorphism on $\Cl(\fX_{r}^s)$, and thus we can compare their effective cones:
Let $I^{1,r}=\mathbb{Z}^{r+1}$ equipped with the symmetric bilinear form defined by the diagonal matrix $\mathrm{diag}(1,-1,\cdots,-1)$ with respect to the standard basis
\[
\mathbf{e}_0=(1,0,\cdots,0), ~ \mathbf{e}_1=(0,1,0\cdots,0), ~\cdots, ~ \mathbf{e}_r=(0,\cdots,0,1)
\]
of $\mathbb{Z}^{r+1}$.
There is a canonical isomorphism $I^{1,r}\stackrel{\simeq}{\longrightarrow} \Cl(\fX_{r}^s)$ defined by
\[
\mathbf{e}_0 \longmapsto \cL^s, ~ \mathbf{e}_1 \longmapsto \cE_1^s, ~\cdots, ~ \mathbf{e}_r \longmapsto \cE_r^s.
\]
Hence we obtain a canonical inclusion
\[
\mathrm{Eff}(\fX_{r}^s) \hookrightarrow \Cl(\fX_{r}^s) \cong I^{1,r},
\]
where $\mathrm{Eff}(\fX_{r}^s)$ is the monoid of the classes of effective divisors on $\fX_{r}^s$.
Under this isomorphism, the vector $\mathbf{k}=-3 \mathbf{e}_0 + \mathbf{e}_1 +\cdots+ \mathbf{e}_r$ corresponds to the canonical class of $\fX_{r}^s$.
\end{Rmk}


The main interest about $\mathrm{Eff}(\fX_{r}^s)$ is whether it is finitely generated or not.
When $r=0$, $\fX_{r}^s\cong\P^2$ and hence the effective cone is generated by the class of the line $\cL^s$.
If $r=1$, the effective cone is generated by two classes: $\cL^s-\cE_1^s$ and $\cE_1^s$.
Now we focus on $r\ge2$ cases.
The result is well-known for smooth del Pezzo surfaces.
Precisely, if $\fX_{r}^s$ is a smooth del Pezzo surface the effective cone is generated by the classes of negative curves if $r\le 7$, and by the classes of negative curves and the anticanonical class if $r=8$ (\cite[Proposition 5.2.2.1]{IA}).
Over $\mathbb{Q}$, it is generated by the classes of negative curves (\cite[Proposition 5.2.1.10]{IA}).
The similar result also holds for any $r$-essentially distinct points on a conic (\cite[Lemma III.i.1]{BH2}).
Also \cite[Proposition 4.1]{GHM} proves the finitely generatedness of $\overline{\mathrm{Eff}}(\fX_{r}^s)$ over $\mathbb{Q}$ for any distinct $r\le 8$ points.
We now prove the finitely generatedness of $\mathrm{Eff}(\fX_{r}^s)$ over $\mathbb{Z}$ for any $r$-essentially distinct points with $2\le r\le8$.

Let $X=\fX_{r}^s$ with $2\le r\le 8$. 
Define the followings (cf. \cite{GHM}, \cite{GH1}).
\begin{itemize}
\item $\mathrm{NEG}(X)=\{ C\in\Cl(X) \,|\, C^2<0, \, C \textrm{ is a prime divisor} \}$
\item $\mathscr{B}_r=\{ \mathbf{e}_i \,|\, 1\le i\le r \}$ ($\mathscr{B}$ is for \textit{blow up} of a point)
\item $\mathscr{V}_r=\{ \mathbf{e}_{i_1}-\mathbf{e}_{i_2}-\cdots-\mathbf{e}_{i_s} \,|\, 2\le s\le r, \, 1\le i_1 < \cdots < i_s\le r \}$ ($\mathscr{V}$ is for \textit{vertical})
\item $\mathscr{L}_r=\{ \mathbf{e}_{0}-\mathbf{e}_{i_1}-\cdots-\mathbf{e}_{i_s} \,|\, 2\le s\le r, \, 1\le i_1 < \cdots < i_s\le r \}$ ($\mathscr{L}$ is for \textit{line})
\item $\mathscr{Q}_r=\{ 2\mathbf{e}_{0}-\mathbf{e}_{i_1}-\cdots-\mathbf{e}_{i_s} \,|\, 5\le s\le r, \, 1\le i_1 < \cdots < i_s\le r \}$ ($\mathscr{Q}$ is for \textit{quadric})
\item $\mathscr{C}_r=\{ 3\mathbf{e}_{0}-2\mathbf{e}_{i_1}-\cdots-\mathbf{e}_{i_s} \,|\, 7\le s\le r, \, 1\le i_1, \cdots, i_s\le r \textrm{ are distinct} \}$ ($\mathscr{C}$ is for \textit{cubic})
\item $\mathscr{M}_8=\{-\mathbf{k} + \mathbf{e}_{0}-\mathbf{e}_{i_1}-\mathbf{e}_{i_2}-\mathbf{e}_{i_3}, \, -\mathbf{k}+ 2\mathbf{e}_{0}-\mathbf{e}_{j_1}-\cdots-\mathbf{e}_{j_6}, \, -2\mathbf{k}-\mathbf{e}_{k} \,|\, 1\le i_1 < i_2 < i_3\le 8, \, 1\le j_1 < \cdots < j_6\le 8, \, 1\le k\le 8\}$ (only for $r=8$)
\item $\mathscr{N}_r = \mathscr{B}_r \cup \mathscr{V}_r \cup \mathscr{L}_r \cup \mathscr{Q}_r \cup \mathscr{C}_r \cup \mathscr{M}_8$
\end{itemize}

\begin{Prop}
$\mathrm{NEG}(X) \subset \mathscr{N}_r$ under the isomorphism in \ref{Marking}.
\end{Prop}

\begin{proof}
We basically follows the proof of \cite[Lemma 2.1]{GH2}.
Let $C$ be a prime divisor on $X$ with negative self-intersection.
Since $L$ is nef, $(C\cdot L)\ge0$.

If $(C\cdot L)=0$, $C=m_1E_1+\cdots+m_rE_r$.
Note that $m_i>0$ for some $i$. If not, $-C\ge0$ and thus $C=0$, a contradiction.
Now $(C\cdot E_i)=-m_i <0$ which implies that $C$ is an irreducible component of the effective divisor $E_i$.
Therefore $C\in \mathscr{B}_r \cup \mathscr{V}_r$.

If $(C\cdot L)=1$, $C$ is a strict transformation of a line on $\P^2$, and hence $C\in\mathscr{L}_r$.
Similarly, If $(C\cdot L)=2$, $C$ is a strict transformation of a quadric, and therefore $C\in\mathscr{Q}_r$.

Suppose $(C\cdot L)\ge3$ and let $D\in|-K_X|$.
If $(C\cdot D)<0$, $C$ is an irreducible component of $D$.
So $D-C\ge0$ and $((D-C)\cdot L)=3-(C\cdot L)\ge0$, hence $(C\cdot L)=3$ and $C\in\mathscr{C}_r$.
On the other hand, if $(C\cdot D)\ge0$, we have $-2\le C^2-(C\cdot D) <0$ by the adjunction formula.
There are only two cases: $C^2=-2$ and $(C\cdot D)=0$, or $C^2=-1$ and $(C\cdot D)=-1$.
Therefore, $C\in \mathscr{N}_r$ (cf. \cite[Proposition 8.2.7]{ID} and \cite[Proposition 8.2.19]{ID}).
\end{proof}

\begin{Prop}\label{FiniteGen}
$\mathrm{Eff}(X)$ is generated by $\mathrm{NEG}(X)$ if $2\le r\le 7$ and by $\mathrm{NEG}(X) \cup \{ -K_X \}$ if $r=8$.
In particular, $\mathrm{Eff}(X)$ is finitely generated.
\end{Prop}

\begin{proof}
Fix an ample divisor $A$ on $X$ and let $D\in|-K_X|$.
Let $G$ be an effective divisor on $X$. We apply the induction on $(G.A)$.
Write $G=G_m+G_f$, where $G_m$ is the moving part and $G_f$ is the fixed part.
Suppose $G_f\ne 0$.
Let $C$ be an irreducible component of $G_f$.
If $C^2\ge0$, $C$ is nef and $(C.D)\ge0$.
Hence $h^2(X,C) = h^0(X,K_X-C)=0$ for $((K_X-C)\cdot L)<0$, and $h^0(C)\ge 1 + \frac{1}{2}(C^2-(K_X\cdot C))$.
If $(C.D)=0$, $C^2< 0$ by Hodge index theorem, a contradiction.
Thus $(C.D)\ge 1$ and $h^0(C)\ge 2$.
Hence $C$ can not be a fixed component of $|G|$.
This implies that $C^2<0$, that is, $G_f$ is a non-negative sum of negative curves on $X$.

On the other hand, suppose $G_m\ne 0$. 
Note that $G_m$ is nef.
By \ref{MainLem1}, the class $G_m = dL-m_1E_1 -\cdots -m_rE_r$ remains nef when we move the points to general position so that the blowing up surface is a smooth del Pezzo surface.
Note that any nef divisor on a smooth del Pezzo surface is effective.
Suppose first $r\le 7$.
Then the new class $G_m'=dL'-m_1E_1'-\cdots -m_rE_r'$ can be written as a non-negative sum of $(-1)$-curves.
So there exists an effective divisor $E$ which is of exceptional class on $X$ such that $G_m-E\ge 0$.
If $E$ is a prime divisor, take $N=E\in\mathrm{NEG}(X)$.
If not, there is $N\in\mathrm{NEG}(X)$ which is an irreducible component of $E$.
Otherwise $E^2\ge0$, a contradiction.
So, $G':=G_m-N$ is effective in both cases.
Note that $(G\cdot A)>(G'\cdot A)\ge0$.
Hence the assertion follows from the induction.
When $r=8$, $G_m'$ is a non-negative sum of $(-1)$-curves and $-K_X$.
Thus, as we have seen before, either $G'=G_m-N$ is effective for some $N\in\mathrm{NEG}(X)$ or $G'=G_m-N$ is effective for some $N\in|-K_X|$.
In both cases, $(G\cdot A)>(G'\cdot A)\ge0$, and applying the induction on $(G\cdot A)$, we conclude our assertion.
\end{proof}

\vspace{1em}
\subsection{Lower-semi continuity theorems}
We always assume $r\le 8$.

\begin{Lem}\label{MainLem2}
Let $\fX_{r}^S\to S$ be an $r$-edpf.
\begin{enumerate}
    \item There exists a non-empty open subset $U\subset S$ such that $\mathrm{Eff}(\fX_{r}^{u_1})=\mathrm{Eff}(\fX_{r}^{u_2})$ for all $u_1,u_2\in U$.
    \item If $D^{u}=d\cL^{u}-m_1\cE_1^{u}-\cdots-m_r\cE_r^{u}\in\mathrm{Eff}(\fX_{r}^{u})$ for some $u\in U$, then $D^u\in\mathrm{Eff}(\fX_{r}^{u})$ for all $u\in U$.
\end{enumerate}
\end{Lem}

\begin{proof}
If $r=0,1$ then every $\fX_{r}^s$ is isomorphic, so we can take $U=S$.
Now assume $r\ge2$ and let $\mathscr{F}^S\in \mathscr{N}_r$.
By the semi-continuity theorem, the set 
\[
B_{\mathscr{F}^S} = \{ s\in S \,|\, \dim_{k} H^0(\fX_{r}^s, \mathscr{F}^s)\ne 0 \}
\]
is closed in $S$. Let
\[
U=S - \bigcup_{\mathscr{F}^S\in \mathscr{N}_r} B_{\mathscr{F}^S}, 
\]
where $\mathscr{F}^S$ runs over the classes in $\mathscr{N}_r$ such that $B_{\mathscr{F}^S}\subsetneq S$.
Let $u_1,u_2\in U$ be arbitrary.
Let $D^{u_1}$ be a class of an effective divisor on $\fX_{r}^{u_1}$.
By \ref{FiniteGen}, $\mathrm{Eff}(\fX_{r}^{u_1})$ is generated by $NEG(\fX_{r}^{u_1})$ and the anticanonical class $-\mathbf{k}$ which is effective.
Thus, we can write
\[
D^{u_1} = \sum_i a_i \alpha_i^{u_1} + b(-\mathbf{k}), ~ a_i,b\in\mathbb{Z}_{\ge0},
\]
for some $\alpha_i^{S}\in \mathscr{N}_r$ with effective $\alpha_i^{u_1}$.
Note that
\[
D^{u_1} = \sum_i a_i \alpha_i^{u_1} + b(-\mathbf{k}) 
= (d\cL^S-m_1\cE_1^S-\cdots-m_r\cE_r^S)|_{\fX_{r}^{u_1}}
\]
for some non-negative integers $d,m_i$.
Denote by $D^{u_2}$ the corresponding divisor for $u_2\in U$, then
\[
D^{u_2} = (d\cL^S-m_1\cE_1^S-\cdots-m_r\cE_r^S)|_{\fX_{r}^{u_2}}
= \sum_i a_i \alpha_i^{u_2} + b(-\mathbf{k}).
\]
By the choice of $U$, $\alpha_i^{u_2}$ is effective, and so is $D^{u_2}$.
This verifies the second assertion.
Furthermore, $D^{u_2}\in\mathrm{Eff}(\fX_{r}^{u_2})$ and hence we have $\mathrm{Eff}(\fX_{r}^{u_1})=\mathrm{Eff}(\fX_{r}^{u_2})$ since $u_1,u_2$ were arbitrary.
\end{proof}


\begin{Cor}\label{MainCor1}
Let $\fX_{r}^S\to S$ and $U\subset S$ be as in \ref{MainLem2}.
Let $\cE^S=m_1\cE_1^S+\cdots+m_r\cE_r^S$ with $m_i\ge0$.
Suppose that $-\cE^s$ satisfies the proximity inequality for all $s\in S$.
\begin{enumerate}
    \item $\widehat{\alpha}(\cE^u)$ is constant for all $u\in U$.
    \item For any $s\in S$, $\widehat{\alpha}(\cE^{s})\le \widehat{\alpha}(\cE^u)$ for all $u\in U$.
\end{enumerate}
\end{Cor}

\begin{proof}
If $u_1,u_2\in U$,
$d\cL^{u_1}-m\cE^{u_1}$ is effective if and only if $d\cL^{u_2}-m\cE^{u_2}$ is effective by \ref{MainLem2}.
Therefore, $\widehat{\alpha}(\cE^{u_1})=\widehat{\alpha}(\cE^{u_2})$, which proves the first assertion.

Let $u_0\in U$ be arbitrary and suppose $d\cL^{u_0}-m\cE^{u_0}$ is effective.
By \ref{MainLem2}, $d\cL^u-m\cE^u$ is effective for all $u\in U$.
Then by \ref{MainLem1}, $d\cL^s-m\cE^s$ is effective for all $s\in S$.
So $\widehat{\alpha}(\cE^{s})\le \frac{d}{m}$ and hence $\widehat{\alpha}(\cE^{s})\le \widehat{\alpha}(\cE^{u})$ for all $u\in U$.
\end{proof}

\begin{Note}\label{ClosedSubschemeZ}
Let $\fX_{r}^S\to S$ and $\cE^S = m_1\cE_1^S+\cdots+m_r\cE_r^S$ with $m_i\ge0$.
Suppose $-\cE^s$ satisfies the proximity inequality for all $s\in S$.
The closed subscheme $Z^S$ defined by $(b_{r,0}^{S})_*\O_{\fX_r}(-\cE^S)$ on $\P^2_{S}$ is flat over $S$, and it commutes with base change (\ref{ClosedSubschemeZ-Lemma}).
We denote by $\cE_Z^S=m_1\cE_1^S+\cdots+m_r\cE_r^S$ in this sense.
Then $(b_{r,0}^s)_*\O_{X_{r}^s}(-\cE_Z^s)$ defines the closed subscheme $Z^{s}$ of $\P^2$, which is a fat point subscheme of $\P^2$ as we discussed in Section \ref{S2}.
\end{Note}

We now state and prove our main results.

\begin{Thm}\label{MainThm1}
Let $\fX_{r}^S\to S$ be an $r$-edpf and $\cE^S=m_1\cE_1^S+\cdots+m_r\cE_r^S$ with $m_i\ge0$.
Suppose $-\cE^s$ satisfies the proximity inequality for all $s\in S$.
Define a function $\widehat{\alpha}_{Z^S}: S \rightarrow \mathbb{R}$ by $s\mapsto \widehat{\alpha}(Z^s)$.
Then every point $s\in S$ is a local minimum of the function $\widehat{\alpha}_{Z^S}$.
\end{Thm}

\begin{proof}
We need to show that for any $s\in S$, there exists an open neighborhood $V\subset S$ of $s$ such that $\widehat{\alpha}(Z^{s})\le \widehat{\alpha}(Z^v)$ for all $v\in V$.
We prove it by induction on $n=\dim(S)$.
If $n=0$, $S$ is a point and the assertion is clear.
Suppose $n\ge1$ and that the assertion holds for all $S'$ of $\dim(S')<n$.
Take an open subset $U\subset S$ with the property as in \ref{MainLem2}.
If $s\in U$, take $V=U$. Then by \ref{MainCor1}, we are done.
So suppose $s\notin U$.
Give an reduced induced scheme structure on $S'=S\setminus U$.
Note that $S'$ is a Noetherian algebraic set of $\dim(S')<n$, 
hence there exists an open neighborhood $V'\subset S'$ of $s$ such that the function $\widehat{\alpha}_{Z^S}$ attains its minimum at $s$ on $V'$ 
by the induction hypothesis.
Take an open subset $V\subset S$ such that $V'=V\cap S'(=V\setminus U)$.
Then for $v\in V\setminus U$, $\widehat{\alpha}(Z^s)\le \widehat{\alpha}(Z^v)$ by the choice of $V'$, and for $v\in V\cap U$, $\widehat{\alpha}(Z^s)\le \widehat{\alpha}(Z^v)$ by \ref{MainCor1}.
Therefore, $\widehat{\alpha}_{Z^S}$ attains its minimum at $s$ on $V$.
\end{proof}

\begin{Thm}\label{MainThm2}
Let $\fX_{r}^S\to S$ be an $r$-edpf and $\cE^S=m_1\cE_1^S+\cdots+m_r\cE_r^S$ with $m_i\ge0$.
Suppose $-\cE^s$ satisfies the proximity inequality for all $s\in S$.
The function $\widehat{\alpha}_{Z^S}: S \rightarrow \mathbb{R}$ defined by $s\mapsto \widehat{\alpha}(Z^s)$ is a lower semi-continuous function on $S$.
Furthermore, the image of $\widehat{\alpha}_{Z^S}$ is a finite set.
\end{Thm}

\begin{proof}
It is lower semi-continuous if and only if the set $S_M:=\{s\in S : \widehat{\alpha}_Z(s)> M \}$ is open for every $M\in\mathbb{R}$.
Let $s_0\in S_M$.
By \ref{MainThm1}, there exists $V\subset S$ a neighborhood of $s_0$ such that $\widehat{\alpha}(Z^v)\ge\widehat{\alpha}(Z^{s_0})$ for all $v\in V$.
Therefore, $V\subset S_M$. It follows that $S_M$ is open in $S$.

For the second part, we proceed by induction on $n=\dim(S)$.
If $n=0$, the assertion is clear.
Suppose $n\ge1$. By \ref{MainCor1}, there is a non-empty open subset $U\subset S$ such that $\widehat{\alpha}(Z^u)$ is constant for all $u\in U$.
Since $S'=S\setminus U$ is a Noetherian algebraic set of $\dim(S')<n$, it has finitely many irreducible components of $\dim<n$.
By the induction hypothesis, the image of $\widehat{\alpha}_Z$ on $S'$ is a finite set.
Therefore, the whole image of $\widehat{\alpha}_Z$ on $S$ is also a finite set.
\end{proof}

We end this section with the remark introduced in Notation \ref{ClosedSubschemeZ}.

\begin{Rmk}\label{ClosedSubschemeZ-Lemma}
Let $\fX_{r}^S\to S$ be an $r$-edpf.
Let $\cE^S=m_1\cE_1^S+\cdots+m_r\cE_r^S$ with $m_i\ge0$ and suppose that $-\cE^s$ satisfies the proximity inequalities for all $s\in S$.
Consider the direct image sheaf $\mathscr{I}^S:=(b_{r,0}^S)_*\O_{\fX_{r}^S}(-\cE^S)$ where $b_{r,0}^S:\fX_{r}^S\to \fX_{0}^S\cong \P^2_S$ is the composition of the blowing ups.
It is a sheaf of ideals on $\P^2_S$, and hence it defines a closed subscheme $Z^S$.
We claim that $Z^S$ is flat over $S$ and commutes with base change, 
that is, $Z^T\cong T\times_S Z^S$ for any $S$-scheme $T$. 
The problem is local, we may assume $S=\Spec A$ and $T=\Spec B$.

For any integer $d\ge0$, there is an exact sequence of flat $\mathcal{O}_S$-modules (\ref{RelEffDiv3}):
\[
0 \longrightarrow \mathcal{O}_{\fX_{r}^S}\left(d\cL^S-\cE^S\right) \longrightarrow \mathcal{O}_{\fX_{r}^S}\left(d\cL^S\right) \longrightarrow \mathcal{O}_{\cE^S}\left(d\cL^S\right) \longrightarrow 0.
\]
Restricting it to a fiber $\fX_{r}^s$, we have an exact sequence of $\mathcal{O}_{\fX_{r}^s}$-modules:
\[
0 \longrightarrow \mathcal{O}_{\fX_{r}^s}\left(d\cL^s-\cE^s\right) \longrightarrow \mathcal{O}_{\fX_{r}^s}\left(d\cL^s\right) \longrightarrow \mathcal{O}_{\cE^s}\left(d\cL^s\right) \longrightarrow 0.
\]
Note that there exists an integer $d_0\ge0$ such that $d\cL^s-\cE^s$ is nef for all $d\ge d_0$ and $s\in S$.
So if $d\ge d_0$, $H^{i}(\fX_{r}^s,\mathcal{O}_{\fX_{r}^s}\left(d\cL^s-\cE^s\right)) = 0$ for all $i\ge1$ (\cite[Theorem 8]{BH5}).
Since $H^i(\fX_{r}^s,\mathcal{O}_{\fX_{r}^s}\left(d\cL^s\right)) =0$ for all $d\ge0$ and $i\ge1$,
by the proper base change theorem for quasi-coherent sheaves (\cite[Theorem III.12.11]{RH}),
\[
H^{i} \left( \fX_{r}^S, \mathcal{O}_{\fX_r^S}\left(d\cL^S-\cE^S\right) \right) = H^{i} \left( \fX_{r}^S, \mathcal{O}_{\fX_r^S}\left(d\cL^S\right) \right) = 0 
\]
if $d\ge d_0$ and $i\ge 1$.
Furthermore, if $d\ge d_0$, we have the following commutative diagram with exact rows:
\[
\xymatrix{
0\ar[r]	& \Gamma \left( \fX_{r}^S, \mathcal{O}_{\fX_{r}^S}\left(d\cL^S-\cE^S\right) \right) \otimes_A k(s) \ar[r] \ar[d]^{\cong}
		& \Gamma \left( \fX_{r}^S, \mathcal{O}_{\fX_r^S}\left(d\cL^S\right) \right) \otimes_A k(s) \ar[d]^{\cong} \\ 
0\ar[r]	& \Gamma \left( \fX_{r}^s, \mathcal{O}_{\fX_{r}^s}\left(d\cL^s-\cE^s\right) \right) \ar[r]
		& \Gamma \left( \fX_{r}^s, \mathcal{O}_{\fX_r^s}\left(d\cL^s\right) \right) 
}
\]
Then by \cite[Proposition III.12.5, Proposition III.12.10]{RH}, if $d\ge d_0$, there is:
\begin{align}
\xymatrix{
0\ar[r]	& \Gamma \left( \fX_{r}^S, \mathcal{O}_{\fX_{r}^S}\left(d\cL^S-\cE^S\right) \right) \otimes_A B \ar[r] \ar[d]^{\cong}
		& \Gamma \left( \fX_{r}^S, \mathcal{O}_{\fX_r^S}\left(d\cL^S\right) \right) \otimes_A B \ar[d]^{\cong}
\\
0\ar[r]	& \Gamma \left( \fX_{r}^T, \mathcal{O}_{\fX_{r}^T}\left(d\cL^T-\cE^T\right) \right) \ar[r]
		& \Gamma\left( \fX_{r}^T, \mathcal{O}_{\fX_r^T}\left(d\cL^T\right) \right)
}
\end{align}
Note that $(b_{r,0}^S)_*\mathcal{O}_{\fX_{r}^S} (d\cL^S-\cE^S)\cong \mathscr{I}^S\otimes_{\O_{\P^2_S}} \O_{\P^2_S}(d) = \mathscr{I}^S(d)$
and that $(b_{r,0}^S)_*\mathcal{O}_{\fX_{r}^S} (d\cL^S)\cong \O_{\P^2_S}(d)$ by projection formula (They also hold for $T$).
Therefore, we have the following commutative diagram with exact rows for $d\ge d_0$:
\[
\xymatrix{
0\ar[r]	& \Gamma \left( \P^2_S, \mathscr{I}^S(d) \right) \otimes_A B \ar[r] \ar[d]^{\cong}
		& \Gamma \left( \P^2_S, \O_{\P^2_S}(d) \right) \otimes_A B \ar[d]^{\cong}
\\
0\ar[r]	& \Gamma \left( \P^2_T, \mathscr{I}^T(d) \right) \ar[r]
		& \Gamma \left( \P^2_T, \O_{\P^2_T}(d) \right)
}
\]
Finally, it yields:
\begin{align}\label{TensbyB}
\xymatrix{
0\ar[r]	& \mathscr{I}^S \otimes_A B \ar[r] \ar[d]^{\cong}
		& \O_{\P^2_S} \otimes_A B \ar[r] \ar[d]^{\cong}
		& \O_{\P^2_S}/\mathscr{I}^S \otimes_A B \ar[r] \ar[d]^{\cong}
		& 0
\\
0\ar[r]	& \mathscr{I}^T \ar[r]
		& \O_{\P^2_T} \ar[r]
		& \O_{\P^2_T}/\mathscr{I}^T \ar[r]
		& 0
}
\end{align}
Hence $Z^T\cong T\times_S Z^S$.
On the other hand, From the diagram \ref{TensbyB}, $\mathrm{Tor}_1^A(\O_{\P^2_S}/\mathscr{I}^S,B)=0$  for any $A$-algebra $B$ since $\O_{\P^2_S}$ is flat over $A$.
Let $M$ be an $A$-module and let $B=A\oplus M$ be the $A$-algebra given by 
\[
(a_1,m_1)\cdot (a_2,m_2) = (a_1a_2, a_2m_1+a_1m_2)
\]
for $(a_1,m_1),\, (a_2,m_2)\in B$.
From the above discussion,
\[
\mathrm{Tor}_1^A(\O_{\P^2_S}/\mathscr{I}^S,B)=\mathrm{Tor}_1^A(\O_{\P^2_S}/\mathscr{I}^S,A)\oplus \mathrm{Tor}_1^A(\O_{\P^2_S}/\mathscr{I}^S,M) =0.
\]
It follows that $\mathrm{Tor}_1^A(\O_{\P^2_S}/\mathscr{I}^S,M)=0$, and therefore, $Z^S$ is flat over $S$.
\end{Rmk}


\vspace{3em}
\section{The Waldschmidt constants for weak del Pezzo surfaces}\label{S5}

For $0\le r \le 8$, there is an open subvariety $\fW_{r-1}\subset \fX_{r-1}$ parametrizing all the $r$-essentially distinct points at whom blowing ups of $\P^2$ are weak del Pezzo surfaces.
We begin with the following proposition.

\begin{Prop}\label{Do8.1.23}\cite[Proposition 8.1.23]{ID}
Let $Y$ be a weak del Pezzo surface.
\begin{enumerate}
    \item Let $f:Y\to\overline{Y}$ be a blowing down of a $(-1)$-curve. Then $\overline{Y}$ is a weak del Pezzo surface.
    \item Let $f:Y'\to Y$ be the blowing up at a point not lying on any $(-2)$-curve. 
    If $K_Y^2>1$ then $Y'$ is a weak del Pezzo surface.
\end{enumerate}
\end{Prop}

Let $Y$ be a blowing up of $\P^2$ at $r$-essentially distinct points.
A class $\alpha\in\Cl(Y)$ is called a \textit{root} if $K_Y\cdot\alpha=0$ and $\alpha^2=-2$.
It is called a \textit{nodal root} if it is a class of a $(-2)$-curve.
A class $\varepsilon\in\Cl(Y)$ is called \textit{exceptional} if $K_Y\cdot\varepsilon=\varepsilon^2=-1$.
Note that if $0\le r\le 8$, there are only finitely many roots and exceptional classes (cf. \cite[Proposition 8.2.7, Proposition 8.2.19]{ID}).
We expand these notions for $r$-edpfs $\fX_{r}^S\to S$.

\begin{Note}
Let $\fX_{r}^S\to S$ be an $r$-edpf and $\mathscr{F}^S$ an invertible sheaf on $\fX_{r}^S$.
$\mathscr{F}^S$ is called a \textit{root} if $(K_{\fX_{r}^s}\cdot\mathscr{F}^s)=0$ and $(\mathscr{F}^s)^2=-2$ for all $s\in S$.
It is called \textit{exceptional} if $(K_{\fX_{r}^s}\cdot\mathscr{F}^s)=(\mathscr{F}^s)^2=-1$ for all $s\in S$.
\end{Note}

Note that since an $r$-edpf is a flat family, the intersection product is invariant on the fibers. 
Hence the condition on the intersection product holds for all points of $S$ if and only if it holds for some points of $S$.

\begin{Thm}\label{UnivFamilyW}
There exists such $i:\fW_{r-1}\hookrightarrow \fX_{r-1}$ for $0\le r \le 8$.
Also, $(\fW_{r-1}, \alpha_{r-1,\bullet}\circ i)$ is the universal object of those $r$-edpf giving a family of weak del Pezzo surfaces.
\end{Thm}

\begin{proof}
$\P^2$ and a blowing up of $\P^2$ at a point have no $(-2)$-curve, 
so $\fW_{r-1}=\fX_{r-1}$ for $r=0,1,2$ by \ref{Do8.1.23}.
For $3\le r\le 8$, let $\mathscr{F}$ be a root on $\fX_{r-1}$ and let $\widehat{\mathscr{F}}=\xi_{r}^*\mathscr{F}-\cE_r$. 
Define
\[
B_{\widehat{\mathscr{F}}} = \{ y_1 \in \fX_{r-1} \,|\, \dim_{k} H^0(\fX_{r}^{y_1}, {\widehat{\mathscr{F}}}^{y_1})\ne 0 \}.
\]
By the semi-continuity theorem, $B_{\widehat{\mathscr{F}}}$ is a closed subset of $\fX_{r-1}$.
Note that if $y_1\in \fX_{r-1}$ corresponds to $r$-distinct points on $\P^2$ in general position, 
then $\dim_k H^0(\fX_{r}^{y_1}, \widehat{\mathscr{F}}^{y_1})= 0$. 
Therefore, $\fX_{r-1}\setminus B_{\widehat{\mathscr{F}}}$ is a non-empty open set, and so is
\[
\fW_{r-1} := \pi_{r-1}^{-1}(\fW_{r-2}) - \bigcup_{\mathscr{F}} B_{\widehat{\mathscr{F}}}
\]
where $\mathscr{F}$ runs over the roots on $\fX_{r-1}$.

Suppose that a fiber $\fX_{r}^x$ over $x\in \fX_{r-1}$ is a weak del Pezzo surface.
Note that $\fX_{r}^x \cong \Bl(\fX_{r-1}^y,x)$ where $y=\pi_{r-1}(x)$.
Hence by \ref{Do8.1.23}, $x\in \pi_{r-1}^{-1}(\fW_{r-2})$.
If $x\in B_{\widehat{\mathscr{F}}}$ for some $\mathscr{F}$, 
there is an effective divisor $C\subset \fX_{r-1}^y$ of class $\mathscr{F}^y$ containing $x$.
Note that every irreducible component of $C$ is a $(-2)$-curve, and thus $x$ lies on a $(-2)$-curve of $\fX_{r-1}^y$.

Conversely, suppose $x\in \fW_{r-1}$.
If $x$ lies on a $(-2)$-curve $C$ on $\fX_{r-1}^y$ where $y=\pi_{r-1}(x)$, then
$x\in B_{\widehat{\mathscr{F}}}$ where $\mathscr{F}$ is the invertible sheaf corresponding to the class of $C$, a contradiction.
Therefore, $x$ does not lie on any $(-2)$-curve on $\fX_{r-1}^y$ and the blowing up surface $\fX_{r}^x$ is a weak del Pezzo surface.
It follows that $\fW_{r-1}$ is the desired one.

The universal property of $(\fW_{r-1},\alpha_{r-1,\bullet}\circ i)$ follows easily from that of $(\fX_{r-1},\alpha_{r-1,\bullet})$.
\end{proof}

Let $X$ be a blowing up of $\P^2$ at $r$-essentially distinct points $p_1,\cdots,p_r$.
Suppose $X$ is a weak del Pezzo surface.
A good point is that $-E_Z = -E_1 - \cdots -E_r$ satisfies the proximity inequalities.
Hence it is natural to consider the fat point subscheme $Z=p_1+\cdots+p_r$ of $\P^2$.

\begin{Prop}\label{Bounds}
Let $X$ be a blowing up of $\P^2$ at $r$-essentially distinct points $p_1,\cdots,p_r$, which is a weak del Pezzo surface of degree $9-r$, and $Z=p_1+\cdots+p_r$ a fat point subscheme of $\P^2$.
Let $\gamma$ be the value of $\widehat{\alpha}(Z)$ when $p_1,\cdots,p_r$ are in general position.
Then $\frac{r}{3} \le \widehat{\alpha}(Z) \le \gamma$ and there are only finitely many possible values for $\widehat{\alpha}(Z)$.
\end{Prop}

\begin{proof}
The upper bound of $\widehat{\alpha}(Z)$ is obtained by \ref{MainCor1} applied to $S=\fW_{r-1}$.
Since $K_X$ is nef, for any effective divisor $dL-mE_Z$, we have $(dL-mE_Z) \cdot K_X = 3d-mr\ge 0$, that is, $\frac{d}{m} \ge \frac{r}{3}$.
Hence, we obtain the lower bound for $\widehat{\alpha}(Z)$.
Finally, \ref{MainThm2} deduces the last part.
\end{proof}

Using the lower semi-continuity of the Waldschmidt constants, we can easily calculate the Waldschmidt constant of $Z=p_1+\cdots+p_r$ for a weak del Pezzo surface $X$.
For example, we calculate the Waldschmidt constant for $r=5$.

\begin{Thm}\label{Deg4}
Let $X$ be a blowing up of $\P^2$ at $5$-essentially distinct points $p_1,\cdots,p_5$, which is a weak del Pezzo surface of degree 4.
Let $Z=p_1+\cdots+p_5$ be a fat point subschem of $\P^2$.
Then $\widehat{\alpha}(Z)=2$, $\frac{9}{5}$, $\frac{7}{4}$ or $\frac{5}{3}$.
\end{Thm}

To calculate the Waldschmidt constant, we recall the following (cf. \cite[Proposition 1.4.8]{BH1}).

\begin{Lem}\label{ZariskiDecomp}
Let $X$ be a blowing up of $\P^2$ at $r$-essentially distinct point.
Let $Z=m_1p_1+\cdots+m_rp_r$ satisfying the proximity inequality.
Suppose that $D=dL-mE_Z$ is effective on $X$, and that $D\cdot F=0$ for a nonzero nef divisor $F$ on $X$.
Then $\widehat{\alpha}(Z)=\frac{d}{m}$.
\end{Lem}

\begin{proof}
By definition, $\widehat{\alpha}(Z)\le \frac{d}{m}$.
For any effective divisor $d'L-m'E_Z$, we have $(d'L-m'E_Z)\cdot F \ge 0$, and thus
\[
\frac{d'}{m'} \ge \frac{(E_Z\cdot F)}{(L\cdot F)} = \frac{d}{m}.
\]
Therefore, $\widehat{\alpha}(Z)\ge \frac{d}{m}$.
\end{proof}

Let $r\ge3$ and consider the lattice $I^{1,r}$ with $\mathbf{k}_r=-3\mathbf{e}_0+\mathbf{e}_1+\cdots+\mathbf{e}_r$.
The sublattice $\mathbb{E}_r=(\mathbb{Z}\mathbf{k}_r)^{\perp}$ of $I^{1,r}$ is called the \textit{$\mathbb{E}_r$-lattice}.
Let $\mathrm{O}(I^{1,r})$ be the orthogonal group of $I^{1,r}$, and $W(\mathbb{E}_r)=\mathrm{O}(I^{1,r})_{\mathbf{k}_r}$ its stabilizer subgroup of $\mathbf{k}_r$, called the \textit{Weyl group} of $\mathbb{E}_r$.
Let 
\[
\alpha_1=\mathbf{e}_0-\mathbf{e}_1-\mathbf{e}_2-\mathbf{e}_3 ~\textrm{and} ~\alpha_i=\mathbf{e}_{i-1}-\mathbf{e}_i, ~ i=2,\cdots,r.
\]
Here, $\alpha_i$ are roots in $\mathbb{E}_r$.
Note that the reflections defined by
\[
r_{\alpha_i} : v \mapsto v+(v\cdot\alpha_i)\alpha_i, ~ i=1,\cdots,r
\]
generates $W(\mathbb{E}_r)$ (cf. \cite[Definition 7.5.8, Corollary 8.2.15]{ID}).
Also, the subgroup $\mathfrak{S}_r\subset W(\mathbb{E}_r)$ generated by $r_{\alpha_i}$, $i=2,\cdots,r$ acts as the permutation group of the vectors $\mathbf{e}_1,\cdots,\mathbf{e}_r$.

Let $X$ be a blowing up of $\P^2$ at $r$-essentially distinct points, which is a weak del Pezzo surface of degree $9-r$.
There is a canonical isomorphism $\phi:I^{1,r} \to \Cl(X)$ as in \ref{Marking}.
Let $L, E_1, \cdots, E_r$ be the divisor classes corresponding to the standard basis of $I^{1,r}$.
Any such isomorphism is called a \textit{marking} on $X$.
It is called a \textit{geometric marking} on $X$ if the marking is canonically obtained by a blowing down structure on $X$.
Similarly to $\mathbb{E}_r$, define the Weyl group of $X$ by $W(X)=\mathrm{O}(\Cl(X))_{K_X}$.
The geometric marking $\phi$ induces an isomorphism $W(\mathbb{E}_r) \to W(X)$. 
Let $\psi:I^{1,r}\to\Cl(X)$ be another geometric marking on $X$ and $L',E_1',\cdots,E_r'$ be the corresponding basis on $\Cl(X)$.
Since $W(X)$ acts simply transitively on the set of markings on $X$ (cf. \cite[Theorem 8.2.12, Corollary 8.2.15]{ID}), there is an $\omega\in W(X)$ such that $\psi=\omega\circ\phi$.
Note that $\omega$ fixes the class $L=\phi(\mathbf{e}_0)$ if and only if $\omega\in\mathfrak{S}_r$.
Therefore, in this case, $\omega$ fixes the class $dL-mE_Z$.
It follows that if $Z=p_1+\cdots+p_r$ and $Z'=p_1'+\cdots+p_r'$ are fat point subschemes of $\P^2$ under the geometric markings $\phi$ and $\psi$, respectively, we have $\widehat{\alpha}(Z)=\widehat{\alpha}(Z')$.
So, $\widehat{\alpha}(Z)$ does not depend on the blowing up order of $p_1,\cdots,p_r$.


Now, back to the situation of \ref{Deg4}, we obtain the following list of types of singularities of a singular del Pezzo surfaces of degree $4$ by analyzing root bases in $\mathbb{E}_5$:
\[
\begin{array}{ll}
    (\nu=5) & D_5, \, 2A_1A_3 \\
    (\nu=4) & D_4, \, A_4, \, A_1A_3, \, 2A_1A_2, \, 4A_1 \\
    (\nu=3) & A_3, \, A_1A_2, \, 3A_1 \\
    (\nu=2) & A_2, \, 2A_1 \\
    (\nu=1) & A_1
\end{array}
\]
where $\nu$ denotes the order of each root basis $\{ \beta_1,\cdots,\beta_\nu \}$ in $\mathbb{E}_5$ (cf. \cite[Section 8.6.3]{ID}).
Let $n$ be the number of points on $\P^2$ among the $5$-essentially distinct points $p_1,\cdots,p_5$,
$\sigma$ the type of the singularities of the anti-canonical model of $X$, 
and $l$ the number of the lines on $X$.
Up to the action of the Weyl group $W(\mathbb{E}_5)$, 
we can completely classify the blowing up models of weak del Pezzo surfaces of degree $4$ by the triple $(n,\sigma,l)$ (cf. \cite[Table 8.6]{ID} or \cite[Proposition 6.1]{CT}).
All the types are listed below by $(n,\sigma,l)(k)$, where $k$ denotes the number of distinct blowing up models if exist:
\[
\begin{array}{ll}
    (n=1) & (1,D_5,1), \, (1,A_4,3); \\
    (n=2) & (2,2A_1A_3,2), \, (2,D_4,2), \, (2,A_4,3)(2), \, (2,A_1A_3,3), \, (2,2A_1A_2,4), \, (2,A_3,5), \\
	    &(2,A_1A_2,6); \\
    (n=3) & (3,A_1A_3,3), \, (3,2A_1A_2,4), \, (3,4A_1,4), \, (3,A_3,4), \, (3,A_3,5)(2), \, (3,A_1A_2,6)(2), \\
	    & (3,3A_1,6), \, (3,A_2,8), \, (3,2A_1,9); \\
    (n=4) & (4,A_1A_2,6), \, (4,3A_1,6), \, (4,A_2,8), \, (4,2A_1,8), \, (4,2A_1,9), \, (4,A_1,12); \\
    (n=5) & (5,2A_1,9), \, (5,A_1,12), \, (5,\emptyset,16).
\end{array}
\]
The corresponding configurations of irreducible roots and lines are listed in Appendix A of the arXiv version of this paper: \url{https://arxiv.org/abs/1802.09755}.

The basic strategy to calculate $\widehat{\alpha}(Z)$ is as follows.
For general cases, $\widehat{\alpha}(Z)=2$.
In fact, the divisor class $D=2L-E_Z$ is effective.
Also the only negative curves are $E_1,\cdots,E_5$, and hence $F=5L-2E_Z$ is nef.
Since $D\cdot F=0$, we conclude that $\widehat{\alpha}(Z)=2$ by \ref{ZariskiDecomp}.
Now suppose that $p_1,\cdots,p_5$ are not in general position.
By \ref{Bounds}, $\widehat{\alpha}(Z)\le2$.
We define the following.
\begin{Def}
Let $X$ be a blowing up of $\P^2$ at $r$-essentially distinct points.
An $r$-edpf $\fX_{r}^T\to T$ over a non-singular curve $T$ is called a \textit{simple $r$-edpf} for $X$ if
\begin{enumerate}
    \item $\fX_{r}^{t_1}\cong X$ as marked surfaces for some $t_1\in U$, where $U\subset T$ is the open subset as in \ref{MainLem2}.
    \item $\pi_{r-1}^T: \fX_{r-1}^T \to T$ is isomorphic to the trivial family $T\times Y\to T$, where $Y$ is the blowing down of $X$ contracting $E_{r}$.
\end{enumerate}
\end{Def}
If there is a simple $5$-edpf $\fX_{6}^T\to T$ for $X$, for any $t_0\notin U$,
\[
\widehat{\alpha}\left( Z^{t_0} \right) \le \widehat{\alpha}\left( Z^{t_1} \right)=\widehat{\alpha}(Z) \le 2.
\]
If $\widehat{\alpha}(Z^{t_0})=2$, we also have $\widehat{\alpha}(Z^{t_1})=2$.
If there is no such $5$-edpf, we directly calculate $\widehat{\alpha}(Z)$ by finding $D$ and $F$ in \ref{ZariskiDecomp}.

\begin{Note}
Denote by $E_{ij}$ the effective divisor of the class $E_i-E_j$, by $L_{ij}$ that of the class $L-E_i-E_j$ and by $L_{ijk}$ that of the class $L-E_i-E_j-E_k$ for $1\le i<j<k \le 5$.
Denote also by $Q$ the effective divisor of the class $2L-E_1-E_2-E_3-E_4-E_5$.
\end{Note}

\begin{proof}[Proof of \ref{Deg4}]
Note that $\frac{5}{3}\le \widehat{\alpha}(Z) \le 2$ by \ref{Bounds}. 
First, there are only two cases with $\nu=5$: $(1,D_5,1)$ and $(2,2A_1A_3,2)$.
For $(1,D_5,1)$,
\[
D= 5L-3E_Z= 5 L_{123} + 2 E_{12} + 4E_{23} + 6E_{34} + 3E_{45}
\]
is effective.
Also $F =-K_X$ is nef and $D\cdot F=0$, hence $\widehat{\alpha}(Z)=\frac{5}{3}$.
Similarly, $\widehat{\alpha}(Z)=\frac{5}{3}$ for $(2,2A_1A_3,2)$.
We take $D=5L-3E_Z$ and $F=-K_X$.
In fact
\[
5L-3E_Z = 2L_{123} + 3L_{145} + 2E_{12} + E_{23}
\]
is effective.

For remaining cases ($\nu\le 4$), we proceed inductively on $n=1,\cdots,5$.

\vspace{0,5em} \noindent \textbf{Case 1(\boldmath$n$=1).}
There is only one case: $(1,A_4,3)$.
In this case, we have $\widehat{\alpha}(Z)=2$.
Here we take $D=Q$ and $F=5L-2E_Z$.

\vspace{0,5em} \noindent \textbf{Case 2(\boldmath$n$=2).}
There are five cases with $\nu=4$: $(2,D_4,2)$, $(2,A_4,3)(a)$, $(2,A_4,3)(b)$, $(2,A_1A_3,3)$ and $(2,2A_1A_2,4)$.
For type $(2,A_4,3)(a)$, we take 
\[
D=7L-4E_Z = 6L_{124} + L_{45} + 2E_{12} + 4E_{23} + 3E_{45}
\]
and $F=4L-E_1-E_2-E_3-2E_4-2E_5$. In this case, we have $\widehat{\alpha}(Z)=\frac{7}{4}$.
For the other cases with $\nu=4$, we take $D=Q$ and
\begin{align*}
F=\, 
& 2L-E_2-E_3-E_4-E_5, \\
& 3L-E_1-2E_2-E_3-E_4-E_5, \\
& 4L-2E_1-2E_2-2E_3-E_4-E_5, ~ \mathrm{and} \\
& 4L-2E_1-E_2-E_3-2E_4-2E_5,
\end{align*}
respectively.
Thus, $\widehat{\alpha}(Z)=2$ in these cases.

The remaining types are $(2,A_3,5)$ and $(2,A_1A_2,6)$.
For those types, we construct simple $5$-edpfs.
Basically, we use the following construction:
Let $Y$ be the blowing down of $X$ contracting $E_5$, and choose a smooth curve $T\subset Y$ passing through $p_5$.
Let $X_5^T=T\times Y \to T$ be the trivial family.
Define a section $\Delta:T\to T\times T\subset X_5^T$ by the diagonal map.
Blowing up $X_5^T$ along $\Delta$, we get a $5$-edpf $\pi_5^T:X_6^T\to T$ with $X=X_6^{p_5}$.
Let $U\subset T$ be the open subset as in \ref{MainLem2}.
If $T$ was properly chosen so that $p_5\in U$, then $\pi_5^T:X_6^T\to T$ is a simple $5$-edpf for $X$, and we have $\widehat{\alpha}(Z)= \widehat{\alpha}(Z^{p_5})\ge \widehat{\alpha}(Z^{t})$ for any $t\in T$.

\noindent \textbf{$\bullet$ Type \boldmath{$(2,A_3,5)$}.}
There is a unique smooth rational curve in the linear system $|Q|$.
Let $Q'$ be its image in $Y$. It is a smooth rational curve. 
We take $T=Q'$ and $p_5'$ to be the intersection point of $T$ and $E_4$.
Then, we get a simple $5$-edpf $\pi_5^T:X_6^T\to T$ for $X=X_6^{p_5}$.
Therefore, $\widehat{\alpha}(Z^{p_5})\ge \widehat{\alpha}(Z^{p_5'})=2$ since $X_6^{p_5'}$ is of the type $(1,A_4,3)$,
and hence $\widehat{\alpha}(Z)=2$.

\noindent \textbf{$\bullet$ Type \boldmath{$(2,A_1A_2,6)$}.}
We take $T=E_4$ and $p_5'$ to be the intersection point of the effective divisor $L_{14}$ and $T$.
The induced blowing up $\pi_5^T:X_6^T\to T$ is a simple $5$-edpf for $X=X_6^{p_5}$. 
Note that $X_6^{p_5'}$ is of the type $(2,A_1A_3,3)$, and therefore $\widehat{\alpha}(Z)=2$.

\vspace{0,5em} \noindent \textbf{Case 3(\boldmath$n$=3).}
There are three cases with $\nu=4$: $(3,A_1A_3,3)$, $(3,2A_1A_2,4)$, and $(3,4A_1,4)$.
For type $(3,2A_1A_2,4)$, we have $\widehat{\alpha}(Z)=\frac{9}{5}$.
In fact,
\[
D=9L-5E_Z = 4L_{124} + 3L_{345} + 2L_{13} + E_{12} + 2E_{45}
\]
is an effective divisor and $F=5L-2E_1-2E_2-3E_3-E_4-E_5$ is a nef divisor such that $D\cdot F=0$.
For the remaining two cases, we take $D=Q$ and $F= 2L-E_2-E_3-E_4-E_5$.
Therefore, $\widehat{\alpha}(Z)=2$ in both cases.

It remains eight cases with $\nu\le3$: $(3,A_3,4)$, $(3,A_3,5)(2)$, $(3,A_1A_2,6)(2)$, $(3,3A_1,6)$, $(3,A_2,8)$, and $(3,2A_1,9)$.
If $X$ is of the type $(3,A_3,5)(a)$, we take $D=2L-E_Z = 2L_{123} + E_{14} + E_{25} + E_3$ and $F=2L-E_1-E_2-E_4-E_5$.
Therefore, $\widehat{\alpha}(Z)=2$ in this case.

We construct simple $5$-edpfs for the remaining types.

\noindent \textbf{$\bullet$ Type \boldmath{$(3,A_3,4)$}.}
Take $T=L_{14}$ and let $p_5'$ be the intersection point of $T$ and $E_4$.
Then $X_6^{p_5'}$ is of the type $(2,A_1A_3,3)$ and hence $\widehat{\alpha}(Z)=2$.

\noindent \textbf{$\bullet$ Types \boldmath{$(3,A_3,5)(b)$} and \boldmath{$(3,A_1A_2,6)(a)$}.}
Let $T$ be a general cubic of the class $3L-E_1-E_2-E_3-E_4$ passing through $p_5$, and let $p_5'$ be the intersection point of the cubic $T$ and $E_4$.
Then $X_6^{p_5'}$ is of the type $(2,D_4,2)$ and $(2,A_1A_3,3)$, respectively. 
Therefore, $\widehat{\alpha}(Z)=2$ for both types.

\noindent \textbf{$\bullet$ Type \boldmath{$(3,A_1A_2,6)(b)$}.}
Let $T$ be the image of $L_{45}$ in $Y$ and $p_5'$ be the intersection point of $T$ and $E_4$.
Then $X_6^{p_5'}$ is of the type $(2,2A_1A_2,4)$ and hence $\widehat{\alpha}(Z)=2$.

\noindent \textbf{$\bullet$ Type \boldmath{$(3,3A_1,6)$}.}
Take $T=L_{34}$ and let $p_5'$ be the intersection point of $T$ and $E_4$.
Then $X_6^{p_5'}$ is of the type $(2,2A_1A_2,4)$ and hence $\widehat{\alpha}(Z)=2$.

\noindent \textbf{$\bullet$ Types \boldmath{$(3,A_2,8)$} and \boldmath{$(3,2A_1,9)$}.}
Let $T$ be the image of $Q$ in $Y$ and $p_5'$ be the intersection point of $T$ and $E_4$.
Then $X_6^{p_5'}$ is of the type $(2,A_1A_2,6)$ in both types, and therefore $\widehat{\alpha}(Z)=2$.

\vspace{0,5em} \noindent \textbf{Case 4(\boldmath$n$=4).}
There are six cases with $\nu\le3$: $(4,A_1A_2,6)$, $(4,3A_1,6)$, $(4,A_2,8)$, $(4,2A_1,8)$, $(4,2A_1,9)$ and $(4,A_1,12)$.

\noindent \textbf{$\bullet$ Types \boldmath{$(4,A_1A_2,6)$}, \boldmath{$(4,3A_1,6)$} and \boldmath{$(4,A_2,8)$}.}
Take $T=L_{14}$ and let $p_5'$ be the intersection point of $T$ and $E_4$.
Then $X_6^{p_5'}$ is of the type $(3,A_1A_3,3)$, $(3,4A_1,4)$ and $(3,A_1A_2,6)$, respectively, and hence $\widehat{\alpha}(Z)=2$ for these types.

\noindent \textbf{$\bullet$ Type \boldmath{$(4,2A_1,8)$}.}
Take $T=L_{34}$ and let $p_5'$ be the intersection point of $T$ and $E_4$.
Then $X_6^{p_5'}$ is of the type $(3,3A_1,6)$ and hence $\widehat{\alpha}(Z)=2$.

\noindent \textbf{$\bullet$ Types \boldmath{$(4,2A_1,9)$} and \boldmath{$(4,A_1,12)$}.}
Let $T$ be the image of $L_{45}$ in $Y$ and $p_5'$ be the intersection point of $T$ and $E_4$.
Then $X_6^{p_5'}$ is of the type $(3,3A_1,6)$ and $(3,2A_1,9)$, respectively, and therefore $\widehat{\alpha}(Z)=2$ in both types.

\vspace{0,5em} \noindent \textbf{Case 5(\boldmath$n$=5).}
Except the general case, we only left two cases: $(5,2A_1,9)$ and $(5,A_1,12)$.
For each case, we construct a simple $5$-edpf for $X$ by taking $T=L_{14}$.
Let $p_5'$ be the intersection point of $T$ and $E_4$.
Then $X_6^{p_5'}$ is of the type $(4,3A_1,6)$ and $(2,2A_1,9)$, respectively, and therefore $\widehat{\alpha}(Z)=2$ in both types.
\end{proof}


\Urlmuskip=0mu plus 1mu\relax
\bibliographystyle{abbrv}

\addresseshere

\clearpage


\newpage
\appendix
\section{The configurations of irreducible roots and lines}\label{List}

\tikzset{
    bp/.style = {draw,circle,fill=black,minimum size=4pt,inner sep=0pt},
    wp/.style = {draw,circle,fill=white,minimum size=4pt,inner sep=0pt}
}


\begin{figure}[!ht]
\begin{minipage}[c]{0.4\textwidth}
\centering\begin{tikzpicture}[xscale=0.5, yscale=0.5]
\draw 
(-4,0) node[wp] (-40) [label=above:$E_{12}$] {} --
(-2,0) node[wp] (-20) [label=above:$E_{23}$] {} --
(0,0) node[wp] (00) [label=above:$E_{34}$] {} --
(2,0) node[wp] (20) [label=above:$E_{45}$] {} --
(4,0) node[bp] (40) [label=above:$E_{5}$] {}
(0,-2) node[wp] (0-2) [label=right:$L_{123}$] {};
\draw (00) -- (0-2);
\end{tikzpicture}
\caption{(1,$D_5$,1)} \label{fig:(1,D5,1)}
\end{minipage}
\hspace{0.1\textwidth}
\begin{minipage}[c]{0.4\textwidth}
\centering\begin{tikzpicture}[xscale=0.5, yscale=0.5]
\draw 
(0,0) node[wp] (00) [label=above:$E_{12}$] {} --
(2,0) node[wp] (20) [label=above:$E_{23}$] {} --
(4,0) node[wp] (40) [label=above:$E_{34}$] {} --
(6,0) node[wp] (60) [label=above:$E_{45}$] {} --
(8,0) node[bp] (80) [label=above:$E_{5}$] {} --
(10,0) node[bp] (100) [label=above:$Q$] {}
(2,-2) node[bp] (2-2) [label=right:$L_{12}$] {};
\draw (20) -- (2-2);
\end{tikzpicture}
\caption{(1,$A_4$,3)} \label{fig:(1,A4,3)}
\end{minipage}
\end{figure}


\begin{figure}[!ht]
\begin{minipage}[c]{0.4\textwidth}
\centering\begin{tikzpicture}[xscale=0.5, yscale=0.5]
\draw 
(-6,0) node[wp] (-60) [label=above:$E_{45}$] {} --
(-4,0) node[bp] (-40) [label=above:$E_{5}$] {} --
(-2,0) node[wp] (20) [label=above:$L_{145}$] {} --
(0,0) node[wp] (00) [label=above:$E_{12}$] {} --
(2,0) node[wp] (20) [label=above:$E_{23}$] {} --
(4,0) node[bp] (40) [label=above:$E_{3}$] {} --
(6,0) node[wp] (60) [label=above:$L_{123}$] {};
\end{tikzpicture}
\caption{(2,$2A_1A_3$,2)} \label{fig:(2,2A1A3,2)}
\end{minipage}
\hspace{0.1\textwidth}
\begin{minipage}[c]{0.4\textwidth}
\centering\begin{tikzpicture}[xscale=0.5, yscale=0.5]
\draw 
(-4,0) node[bp] (-40) [label=above:$E_{1}$] {} --
(-2,0) node[wp] (-20) [label=above:$L_{123}$] {} --
(0,0) node[wp] (00) [label=above:$E_{34}$] {} --
(2,0) node[wp] (20) [label=above:$E_{45}$] {} --
(4,0) node[bp] (40) [label=above:$E_{5}$] {}
(0,-2) node[wp] (0-2) [label=right:$E_{23}$] {};
\draw (00) -- (0-2);
\end{tikzpicture}
\caption{(2,$D_4$,2)} \label{fig:(2,D4,2)}
\end{minipage}
\end{figure}

\begin{figure}[!ht]
\begin{minipage}[c]{0.4\textwidth}
\centering\begin{tikzpicture}[xscale=0.5, yscale=0.5]
\draw 
(0,0) node[wp] (00) [label=above:$E_{12}$] {} --
(2,0) node[wp] (20) [label=above:$E_{23}$] {} --
(4,0) node[wp] (40) [label=above:$L_{124}$] {} --
(6,0) node[wp] (60) [label=above:$E_{45}$] {} --
(8,0) node[bp] (80) [label=above:$E_{5}$] {} --
(10,0) node[bp] (100) [label=above:$L_{45}$] {}
(2,-2) node[bp] (2-2) [label=right:$E_{3}$] {};
\draw (20) -- (2-2);
\end{tikzpicture}
\caption{(2,$A_4$,3)(a)} \label{fig:(2,A4,3)(a)}
\end{minipage}
\hspace{0.1\textwidth}
\begin{minipage}[c]{0.4\textwidth}
\centering\begin{tikzpicture}[xscale=0.5, yscale=0.5]
\draw 
(0,0) node[wp] (00) [label=above:$L_{134}$] {} --
(2,0) node[wp] (20) [label=above:$E_{45}$] {} --
(4,0) node[wp] (40) [label=above:$E_{34}$] {} --
(6,0) node[wp] (60) [label=above:$E_{13}$] {} --
(8,0) node[bp] (80) [label=above:$L_{12}$] {} --
(10,0) node[bp] (100) [label=above:$E_{2}$] {}
(2,-2) node[bp] (2-2) [label=right:$E_{5}$] {};
\draw (20) -- (2-2);
\end{tikzpicture}
\caption{(2,$A_4$,3)(b)} \label{fig:(2,A4,3)(b)}
\end{minipage}
\end{figure}

\begin{figure}[!ht]
\begin{minipage}[c]{0.4\textwidth}
\centering\begin{tikzpicture}[xscale=0.5, yscale=0.5]
\draw 
(0,2) node[bp] (02) [label=left:$E_{3}$] {} 
(0,-2) node[bp] (0-2) [label=left:$L_{12}$] {} --
(2,0) node[wp] (20) [label=left:$E_{23}$] {} --
(4,0) node[wp] (40) [label=above:$E_{12}$] {} --
(6,0) node[wp] (60) [label=above:$L_{145}$] {} --
(8,0) node[bp] (80) [label=above:$E_{5}$] {} --
(10,0) node[wp] (100) [label=above:$E_{45}$] {};
\draw (20) -- (02);
\end{tikzpicture}\caption{(2,$A_1A_3$,3)} \label{fig:(2,A1A3,3)}
\end{minipage}
\hspace{0.1\textwidth}
\begin{minipage}[c]{0.4\textwidth}
\centering\begin{tikzpicture}[xscale=0.5, yscale=0.5]
\draw 
(0,0) node[wp] (00) [label=below:$E_{12}$] {} --
(2,0) node[wp] (20) [label=below:$E_{23}$] {} --
(4,2) node[bp] (42) [label=right:$E_{3}$] {} --
(4,4) node[wp] (44) [label=right:$L_{123}$] {} --
(2,6) node[bp] (26) [label=above:$L_{45}$] {} --
(0,6) node[bp] (06) [label=above:$E_{5}$] {} --
(-2,4) node[wp] (-24) [label=left:$E_{45}$] {} --
(-2,2) node[bp] (-22) [label=left:$L_{14}$] {};
\draw (00) -- (-22);
\end{tikzpicture}\caption{(2,$2A_1A_2$,4)} \label{fig:(2,2A1A2,4)}
\end{minipage}
\end{figure}

\begin{figure}[!ht]
\begin{minipage}[c]{0.4\textwidth}
\centering\begin{tikzpicture}[xscale=0.5, yscale=0.5]
\draw 
(0,0) node[wp] (00) [label=above:$E_{12}$] {} --
(2,0) node[wp] (20) [label=above:$E_{23}$] {} --
(4,0) node[wp] (40) [label=above:$E_{34}$] {} --
(6,-2) node[bp] (6-2) [label=right:$E_{4}$] {} --
(4,-4) node[bp] (4-4) [label=below:$Q$] {} --
(0,-4) node[bp] (0-4) [label=below:$E_{5}$] {} --
(-2,-2) node[bp] (-2-2) [label=left:$L_{15}$] {}
(2,-2) node[bp] (2-2) [label=left:$L_{12}$] {};
\draw (-2-2) -- (00);
\draw (20) -- (2-2);
\end{tikzpicture}
\caption{(2,$A_3$,5)} \label{fig:(2,A3,5)}
\end{minipage}
\hspace{0.1\textwidth}
\begin{minipage}[c]{0.4\textwidth}
\centering\begin{tikzpicture}[xscale=0.5, yscale=0.5]
\draw
(0,0) node[bp] (00) [label=above:$L_{12}$] {} --
(2,0) node[bp] (20) [label=above:$L_{45}$] {} --
(3,-1) node[bp] (3-1) [label=left:$E_{5}$] {} --
(2,-2) node[bp] (2-2) [label=below:$Q$] {} --
(0,-2) node[bp] (0-2) [label=below:$E_{3}$] {} --
(-1,-1) node[wp] (-1-1) [label=right:$E_{23}$] {} --
(-3,-1) node[wp] (-3-1) [label=above:$E_{12}$] {} --
(1,-5) node[bp] (1-5) [label=below:$L_{14}$] {} --
(5,-1) node[wp] (5-1) [label=above:$E_{45}$] {};
\draw (-1-1) -- (00);
\draw (5-1) -- (3-1);
\end{tikzpicture}
\caption{(2,$A_1A_2$,6)} \label{fig:(2,A1A2,6)}
\end{minipage}
\end{figure}


\begin{figure}[!ht]
\begin{minipage}[c]{0.4\textwidth}
\centering\begin{tikzpicture}[xscale=0.5, yscale=0.5]
\draw 
(0,2) node[bp] (02) [label=left:$E_{2}$] {} 
(0,-2) node[bp] (0-2) [label=left:$E_{3}$] {} --
(2,0) node[wp] (20) [label=left:$L_{123}$] {} --
(4,0) node[wp] (40) [label=above:$E_{14}$] {} --
(6,0) node[wp] (60) [label=above:$E_{45}$] {} --
(8,0) node[bp] (80) [label=above:$E_{5}$] {} --
(10,0) node[wp] (100) [label=above:$L_{145}$] {};
\draw (20) -- (02);
\end{tikzpicture}
\caption{(3,$A_1A_3$,3)} \label{fig:(3,A1A3,3)}
\end{minipage}
\hspace{0.1\textwidth}
\begin{minipage}[c]{0.4\textwidth}
\centering\begin{tikzpicture}[xscale=0.5, yscale=0.5]
\draw 
(0,0) node[wp] (00) [label=below:$E_{45}$] {} --
(2,0) node[wp] (20) [label=below:$L_{124}$] {} --
(4,2) node[bp] (42) [label=right:$E_{2}$] {} --
(4,4) node[wp] (44) [label=right:$E_{12}$] {} --
(2,6) node[bp] (26) [label=above:$L_{13}$] {} --
(0,6) node[bp] (06) [label=above:$E_{3}$] {} --
(-2,4) node[wp] (-24) [label=left:$L_{345}$] {} --
(-2,2) node[bp] (-22) [label=left:$E_{5}$] {};
\draw (00) -- (-22);
\end{tikzpicture}
\caption{(3,$2A_1A_2$,4)} \label{fig:(3,2A1A2,4)}
\end{minipage}
\end{figure}

\begin{figure}[!ht]
\begin{minipage}[c]{0.4\textwidth}
\centering\begin{tikzpicture}[xscale=0.5, yscale=0.5]
\draw 
(0,0) node[bp] (00) [label=below:$e_{2}$] {} --
(2,0) node[wp] (20) [label=below:$e_{12}$] {} --
(4,2) node[bp] (42) [label=right:$l_{14}$] {} --
(4,4) node[wp] (44) [label=right:$e_{45}$] {} --
(2,6) node[bp] (26) [label=above:$e_{5}$] {} --
(0,6) node[wp] (06) [label=above:$l_{345}$] {} --
(-2,4) node[bp] (-24) [label=left:$e_{3}$] {} --
(-2,2) node[wp] (-22) [label=left:$l_{123}$] {};

\draw (00) -- (-22);
\end{tikzpicture}
\caption{(3,$4A_1$,4)} \label{fig:(3,4A1,4)}
\end{minipage}
\hspace{0.1\textwidth}
\begin{minipage}[c]{0.4\textwidth}
\centering\begin{tikzpicture}[xscale=0.5, yscale=0.5]
\draw 
(0,2) node[bp] (02) [label=left:$E_{4}$] {} 
(0,-2) node[bp] (0-2) [label=left:$E_{5}$] {} --
(2,0) node[wp] (20) [label=left:$L_{145}$] {} --
(4,0) node[wp] (40) [label=above:$E_{12}$] {} --
(6,0) node[wp] (60) [label=right:$E_{23}$] {} --
(8,2) node[bp] (82) [label=right:$E_{3}$] {}
(8,-2) node[bp] (8-2) [label=right:$L_{12}$] {};
\draw (20) -- (02);
\draw (60) -- (8-2);
\end{tikzpicture}
\caption{(3,$A_3$,4)} \label{fig:(3,A3,4)}
\end{minipage}
\end{figure}

\begin{figure}[!ht]
\begin{minipage}[c]{0.4\textwidth}
\centering\begin{tikzpicture}[xscale=0.5, yscale=0.5]
\draw 
(0,0) node[wp] (00) [label=above:$E_{14}$] {} --
(2,0) node[wp] (20) [label=above:$L_{123}$] {} --
(4,0) node[wp] (40) [label=above:$E_{25}$] {} --
(6,-2) node[bp] (6-2) [label=right:$E_{5}$] {} --
(4,-4) node[bp] (4-4) [label=below:$L_{25}$] {} --
(0,-4) node[bp] (0-4) [label=below:$L_{14}$] {} --
(-2,-2) node[bp] (-2-2) [label=left:$E_{4}$] {}
(2,-2) node[bp] (2-2) [label=left:$E_{3}$] {};
\draw (-2-2) -- (00);
\draw (20) -- (2-2);
\end{tikzpicture}
\caption{(3,$A_3$,5)(a)} \label{fig:(3,A3,5)(a)}
\end{minipage}
\hspace{0.1\textwidth}
\begin{minipage}[c]{0.4\textwidth}
\centering\begin{tikzpicture}[xscale=0.5, yscale=0.5]
\draw 
(0,0) node[wp] (00) [label=above:$E_{23}$] {} --
(2,0) node[wp] (20) [label=above:$E_{34}$] {} --
(4,0) node[wp] (40) [label=above:$L_{123}$] {} --
(6,-2) node[bp] (6-2) [label=right:$E_{1}$] {} --
(4,-4) node[bp] (4-4) [label=below:$L_{15}$] {} --
(0,-4) node[bp] (0-4) [label=below:$E_{5}$] {} --
(-2,-2) node[bp] (-2-2) [label=left:$L_{25}$] {}
(2,-2) node[bp] (2-2) [label=left:$E_{4}$] {};
\draw (-2-2) -- (00);
\draw (20) -- (2-2);
\end{tikzpicture}
\caption{(3,$A_3$,5)(b)} \label{fig:(3,A3,5)(b)}
\end{minipage}
\end{figure}

\begin{figure}[!ht]
\begin{minipage}[c]{0.4\textwidth}
\centering\begin{tikzpicture}[xscale=0.5, yscale=0.5]
\draw
(0,0) node[bp] (00) [label=above:$L_{35}$] {} --
(2,0) node[bp] (20) [label=above:$E_{5}$] {} --
(3,-1) node[bp] (3-1) [label=left:$L_{15}$] {} --
(2,-2) node[bp] (2-2) [label=below:$L_{34}$] {} --
(0,-2) node[bp] (0-2) [label=below:$E_{4}$] {} --
(-1,-1) node[wp] (-1-1) [label=right:$E_{34}$] {} --
(-3,-1) node[wp] (-3-1) [label=above:$L_{123}$] {} --
(1,-5) node[bp] (1-5) [label=below:$E_{2}$] {} --
(5,-1) node[wp] (5-1) [label=above:$E_{12}$] {};
\draw (-1-1) -- (00);
\draw (5-1) -- (3-1);
\end{tikzpicture}
\caption{(3,$A_1A_2$,6)(a)} \label{fig:(3,A1A2,6)(a)}
\end{minipage}
\hspace{0.1\textwidth}
\begin{minipage}[c]{0.4\textwidth}
\centering\begin{tikzpicture}[xscale=0.5, yscale=0.5]
\draw
(0,0) node[bp] (00) [label=above:$L_{14}$] {} --
(2,0) node[bp] (20) [label=above:$E_{4}$] {} --
(3,-1) node[bp] (3-1) [label=left:$L_{45}$] {} --
(2,-2) node[bp] (2-2) [label=below:$E_{5}$] {} --
(0,-2) node[bp] (0-2) [label=below:$L_{15}$] {} --
(-1,-1) node[wp] (-1-1) [label=right:$E_{12}$] {} --
(-3,-1) node[wp] (-3-1) [label=above:$E_{23}$] {} --
(1,-5) node[bp] (1-5) [label=below:$E_{3}$] {} --
(5,-1) node[wp] (5-1) [label=above:$L_{123}$] {};
\draw (-1-1) -- (00);
\draw (5-1) -- (3-1);
\end{tikzpicture}
\caption{(3,$A_1A_2$,6)(b)} \label{fig:(3,A1A2,6)(b)}
\end{minipage}
\end{figure}

\begin{figure}[!ht]
\begin{minipage}[c]{0.4\textwidth}
\centering\begin{tikzpicture}[xscale=0.5, yscale=0.5]
\draw
(0,0) node[bp] (00) [label=above:$E_{2}$] {} --
(2,0) node[bp] (20) [label=above:$L_{12}$] {} --
(3,-1) node[wp] (3-1) [label=left:$L_{345}$] {} --
(2,-2) node[bp] (2-2) [label=below:$E_{5}$] {} --
(0,-2) node[bp] (0-2) [label=below:$L_{15}$] {} --
(-1,-1) node[wp] (-1-1) [label=right:$E_{12}$] {} --
(-3,-1) node[bp] (-3-1) [label=above:$L_{13}$] {} --
(1,-5) node[wp] (1-5) [label=below:$E_{34}$] {} --
(5,-1) node[bp] (5-1) [label=above:$E_{4}$] {};
\draw (-1-1) -- (00);
\draw (5-1) -- (3-1);
\end{tikzpicture}
\caption{(3,$3A_1$,6)} \label{fig:(3,3A1,6)}
\end{minipage}
\hspace{0.1\textwidth}
\begin{minipage}[c]{0.4\textwidth}
\centering\begin{tikzpicture}[xscale=0.5, yscale=0.5]
\draw 
(0,0) node[bp] (00) [label=above:$L_{15}$] {} --
(2,0) node[bp] (20) [label=above:$E_{5}$] {} --
(4,0) node[bp] (40) [label=above:$L_{45}$] {} --
(6,0) node[bp] (60) [label=above:$L_{12}$] {} --
(8,-2) node[wp] (8-2) [label=right:$E_{23}$] {} --
(6,-4) node[bp] (6-4) [label=below:$E_{3}$] {} --
(4,-4) node[bp] (4-4) [label=below:$Q$] {} --
(2,-4) node[bp] (2-4) [label=below:$E_{4}$] {} --
(0,-4) node[bp] (0-4) [label=below:$L_{14}$] {} --
(-2,-2) node[wp] (-2-2) [label=left:$E_{12}$] {};
\draw (-2-2) -- (00);
\draw (-2-2) -- (8-2);
\draw (20) -- (4-4);
\draw (40) -- (2-4);
\end{tikzpicture}
\caption{(3,$A_2$,8)} \label{fig:(3,A2,8)}
\end{minipage}
\end{figure}

\begin{figure}[!ht]
\centering\begin{tikzpicture}[xscale=0.8, yscale=0.8]
\draw
(0,0) node[bp] (00) [label=above:$L_{15}$] {} --
(2,0) node[bp] (20) [label=above:$L_{34}$] {} --
(4,0) node[bp] (40) [label=above:$E_{4}$] {} --
(6,-1) node[wp] (6-1) [label=above:$E_{34}$] {} --
(4,-2) node[bp] (4-2) [label=above:$L_{35}$] {} --
(2,-2) node[bp] (2-2) [label=below:$L_{12}$] {} --
(0,-2) node[bp] (0-2) [label=above:$E_{2}$] {} --
(-2,-1) node[wp] (-2-1) [label=above:$E_{12}$] {}
(2,-4) node[bp] (2-4) [label=below:$L_{13}$] {}
(0,-1) node[bp] (0-1) [label=left:$E_{5}$] {} --
(4,-1) node[bp] (4-1) [label=right:$Q$] {};
\draw (-2-1) -- (00);
\draw (20) -- (2-2);
\draw (0-1) -- (00);
\draw (0-1) -- (4-2);
\draw (4-1) -- (40);
\draw (4-1) -- (0-2);
\draw (-2-1) -- (2-4);
\draw (2-4) -- (6-1);
\end{tikzpicture}
\caption{(3,$2A_1$,9)} \label{fig:(3,2A1,9)}
\end{figure}


\begin{figure}[!ht]
\begin{minipage}[c]{0.4\textwidth}
\centering\begin{tikzpicture}[xscale=0.5, yscale=0.5]
\draw
(0,0) node[bp] (00) [label=above:$E_{2}$] {} --
(2,0) node[bp] (20) [label=above:$L_{25}$] {} --
(3,-1) node[bp] (3-1) [label=left:$E_{5}$] {} --
(2,-2) node[bp] (2-2) [label=below:$L_{35}$] {} --
(0,-2) node[bp] (0-2) [label=below:$E_{3}$] {} --
(-1,-1) node[wp] (-1-1) [label=right:$L_{123}$] {} --
(-3,-1) node[wp] (-3-1) [label=above:$E_{14}$] {} --
(1,-5) node[bp] (1-5) [label=below:$E_{4}$] {} --
(5,-1) node[wp] (5-1) [label=above:$L_{145}$] {};
\draw (-1-1) -- (00);
\draw (5-1) -- (3-1);
\end{tikzpicture}
\caption{(4,$A_1A_2$,6)} \label{fig:(4,A1A2,6)}
\end{minipage}
\hspace{0.1\textwidth}
\begin{minipage}[c]{0.4\textwidth}
\centering\begin{tikzpicture}[xscale=0.5, yscale=0.5]
\draw
(0,0) node[bp] (00) [label=above:$E_{4}$] {} --
(2,0) node[bp] (20) [label=above:$L_{34}$] {} --
(3,-1) node[wp] (3-1) [label=left:$E_{23}$] {} --
(2,-2) node[bp] (2-2) [label=below:$L_{35}$] {} --
(0,-2) node[bp] (0-2) [label=below:$E_{5}$] {} --
(-1,-1) node[wp] (-1-1) [label=right:$L_{145}$] {} --
(-3,-1) node[bp] (-3-1) [label=above:$L_{13}$] {} --
(1,-5) node[wp] (1-5) [label=below:$L_{123}$] {} --
(5,-1) node[bp] (5-1) [label=above:$E_{2}$] {};
\draw (-1-1) -- (00);
\draw (5-1) -- (3-1);
\end{tikzpicture}
\caption{(4,$3A_1$,6)} \label{fig:(4,3A1,6)}
\end{minipage}
\end{figure}

\begin{figure}[!ht]
\begin{minipage}[c]{0.4\textwidth}
\centering\begin{tikzpicture}[xscale=0.5, yscale=0.5]
\draw 
(0,0) node[bp] (00) [label=above:$L_{13}$] {} --
(2,0) node[bp] (20) [label=above:$E_{3}$] {} --
(4,0) node[bp] (40) [label=above:$L_{34}$] {} --
(6,0) node[bp] (60) [label=above:$E_{4}$] {} --
(8,-2) node[wp] (8-2) [label=right:$L_{145}$] {} --
(6,-4) node[bp] (6-4) [label=below:$E_{5}$] {} --
(4,-4) node[bp] (4-4) [label=below:$L_{35}$] {} --
(2,-4) node[bp] (2-4) [label=below:$L_{12}$] {} --
(0,-4) node[bp] (0-4) [label=below:$E_{2}$] {} --
(-2,-2) node[wp] (-2-2) [label=left:$E_{12}$] {};
\draw (-2-2) -- (00);
\draw (-2-2) -- (8-2);
\draw (20) -- (4-4);
\draw (40) -- (2-4);
\end{tikzpicture}
\caption{(4,$A_2$,8)} \label{fig:(4,A2,8)}
\end{minipage}
\hspace{0.1\textwidth}
\begin{minipage}[c]{0.4\textwidth}
\centering\begin{tikzpicture}[xscale=0.7, yscale=0.7]
\draw 
(0,0) node[wp] (00) [label=left:$E_{12}$] {} --
(2,0.5) node[bp] (2,0.5) [label=above:$L_{23}$] {} --
(6,0.5) node[bp] (6,0.5) [label=above:$E_{3}$] {} --
(8,0) node[wp] (80) [label=right:$L_{345}$] {}
(2,1.5) node[bp] (2,1.5) [label=above:$E_{1}$] {} --
(6,1.5) node[bp] (6,1.5) [label=above:$L_{12}$] {} 
(2,-0.5) node[bp] (2,-0.5) [label=above:$L_{24}$] {} --
(6,-0.5) node[bp] (6,-0.5) [label=above:$E_{4}$] {}
(2,-1.5) node[bp] (2,-1.5) [label=above:$L_{25}$] {} --
(6,-1.5) node[bp] (6,-1.5) [label=above:$E_{5}$] {};
\draw (00) -- (2,1.5);
\draw (00) -- (2,-0.5);
\draw (00) -- (2,-1.5);
\draw (80) -- (6,1.5);
\draw (80) -- (6,-0.5);
\draw (80) -- (6,-1.5);
\end{tikzpicture}
\caption{(4,$2A_1$,8)} \label{fig:(4,2A1,8)}
\end{minipage}
\end{figure}

\begin{figure}[!ht]
\centering\begin{tikzpicture}[xscale=0.8, yscale=0.8]
\draw
(0,0) node[bp] (00) [label=above:$L_{14}$] {} --
(2,0) node[bp] (20) [label=above:$L_{35}$] {} --
(4,0) node[bp] (40) [label=above:$E_{3}$] {} --
(6,-1) node[wp] (6-1) [label=above:$L_{123}$] {} --
(4,-2) node[bp] (4-2) [label=above:$L_{45}$] {} --
(2,-2) node[bp] (2-2) [label=below:$E_{5}$] {} --
(0,-2) node[bp] (0-2) [label=above:$L_{15}$] {} --
(-2,-1) node[wp] (-2-1) [label=above:$E_{12}$] {}
(2,-4) node[bp] (2-4) [label=below:$E_{2}$] {}
(0,-1) node[bp] (0-1) [label=left:$E_{4}$] {} --
(4,-1) node[bp] (4-1) [label=right:$L_{34}$] {};
\draw (-2-1) -- (00);
\draw (20) -- (2-2);
\draw (0-1) -- (00);
\draw (0-1) -- (4-2);
\draw (4-1) -- (40);
\draw (4-1) -- (0-2);
\draw (-2-1) -- (2-4);
\draw (2-4) -- (6-1);
\end{tikzpicture}
\caption{(4,$2A_1$,9)} \label{fig:(4,2A1,9)}
\end{figure}

\begin{figure}[!ht]
\centering\begin{tikzpicture}[xscale=0.8, yscale=0.8]
\draw 
(0,0) node[wp] (00) [label=45:$E_{12}$] {} --
(2,0) node[bp] (20) [label=above:$E_{2}$] {} --
(4,0) node[bp] (40) [label=right:$L_{12}$] {} --
(2,2) node[bp] (22) [label=right:$L_{34}$] {} --
(0,4) node[bp] (04) [label=left:$E_{3}$] {} --
(0,2) node[bp] (02) [label=left:$L_{13}$] {} --
(0,-2) node[bp] (0-2) [label=right:$L_{15}$] {} --
(0,-4) node[bp] (0-4) [label=right:$E_{5}$] {} --
(2,-2) node[bp] (2-2) [label=right:$L_{35}$] {}
(-2,0) node[bp] (-20) [label=above:$L_{14}$] {}
(-4,0) node[bp] (-40) [label=left:$E_{4}$] {} --
(-2,-2) node[bp] (-2-2) [label=below:$L_{45}$] {}
(-2,2) node[bp] (-22) [label=left:$Q$] {};
\draw (00)-- (-20) -- (-40) -- (-22);
\draw (22) -- (-40);
\draw (22) -- (0-2);
\draw (2-2) -- (04);
\draw (2-2) -- (40);
\draw (2-2) -- (-20);
\draw (-22) -- (04);
\draw (-22) -- (0-4);
\draw (-22) -- (20);
\draw (-2-2) -- (40);
\draw (-2-2) -- (0-4);
\draw (-2-2) -- (02);
\end{tikzpicture}
\caption{(4,$A_1$,12)} \label{fig:(4,A1,12)}
\end{figure}


\begin{figure}[t]
\begin{center}
\begin{tikzpicture}[xscale=0.8, yscale=0.8]
\draw
(0,0) node[bp] (00) [label=above:$E_{2}$] {} --
(2,0) node[bp] (20) [label=above:$L_{24}$] {} --
(4,0) node[bp] (40) [label=above:$E_{4}$] {} --
(6,-1) node[wp] (6-1) [label=above:$L_{145}$] {} --
(4,-2) node[bp] (4-2) [label=above:$E_{5}$] {} --
(2,-2) node[bp] (2-2) [label=below:$L_{35}$] {} --
(0,-2) node[bp] (0-2) [label=above:$E_{3}$] {} --
(-2,-1) node[wp] (-2-1) [label=above:$L_{123}$] {}
(2,-4) node[bp] (2-4) [label=below:$E_{1}$] {}
(0,-1) node[bp] (0-1) [label=left:$L_{25}$] {} --
(4,-1) node[bp] (4-1) [label=right:$L_{34}$] {}
;

\draw (-2-1) -- (00);
\draw (20) -- (2-2);
\draw (0-1) -- (00);
\draw (0-1) -- (4-2);
\draw (4-1) -- (40);
\draw (4-1) -- (0-2);
\draw (-2-1) -- (2-4);
\draw (2-4) -- (6-1);
\end{tikzpicture}
\end{center}
\caption{(5,$2A_1$,9)} \label{fig:(5,2A1,9)}
\end{figure}

\begin{figure}[t]
\begin{center}
\begin{tikzpicture}[xscale=0.8, yscale=0.8]
\draw 
(0,0) node[wp] (00) [label=45:$L_{145}$] {} --
(2,0) node[bp] (20) [label=above:$E_{4}$] {} --
(4,0) node[bp] (40) [label=right:$L_{34}$] {} --
(2,2) node[bp] (22) [label=right:$L_{25}$] {} --
(0,4) node[bp] (04) [label=left:$L_{13}$] {} --
(0,2) node[bp] (02) [label=left:$E_{1}$] {} --
(0,-2) node[bp] (0-2) [label=right:$E_{5}$] {} --
(0,-4) node[bp] (0-4) [label=right:$L_{35}$] {} --
(2,-2) node[bp] (2-2) [label=right:$E_{3}$] {} -- (40)
(-2,0) node[bp] (-20) [label=above:$L_{23}$] {} -- (00)
(-4,0) node[bp] (-40) [label=left:$E_{2}$] {} --
(-2,-2) node[bp] (-2-2) [label=below:$L_{12}$] {} -- (0,-4)
(-2,2) node[bp] (-22) [label=left:$L_{24}$] {} -- (04);

\draw (-20) -- (-40) -- (-22);
\draw (22) -- (-40);
\draw (22) -- (0-2);
\draw (2-2) -- (04);
\draw (2-2) -- (-20);
\draw (-22) -- (0-4);
\draw (-22) -- (20);
\draw (-2-2) -- (40);
\draw (-2-2) -- (02);
\end{tikzpicture}
\end{center}
\caption{(5,$A_1$,12)} \label{fig:(5,A1,12)}
\end{figure}


\end{document}